\documentclass[12pt]{article}

\usepackage{amsmath,amsthm,amsfonts,amssymb}
\usepackage{tikz}
\usetikzlibrary{angles,quotes}

\theoremstyle{plain}
\newtheorem{theorem}{Theorem}[section]
\newtheorem{proposition}[theorem]{Proposition}
\newtheorem{lemma}[theorem]{Lemma}
\newtheorem{corollary}[theorem]{Corollary}

\theoremstyle{definition}
\newtheorem{definition}[theorem]{\it Definition}
\newtheorem{algorithm}[theorem]{\it Algorithm}
\newtheorem{example}[theorem]{\it Example}

\theoremstyle{remark}
\newtheorem{remark}[theorem]{Remark}

\begin{document}

\title{Monotonicity equivalence and synchronizability for a system of probability distributions}

\author{
Motoya Machida \\
{\normalsize Department of Mathematics, Tennessee Technological University} \\
{\normalsize Cookeville, TN 38505, USA}
}

\date{\today}

\maketitle

\begin{abstract}
A system $(P_\alpha: \alpha\in\mathcal{A})$ of probability distributions
on a partially ordered set (poset) $\mathcal{S}$ indexed by another poset
$\mathcal{A}$
can be realized by a system of
$\mathcal{S}$-valued random variables $X_\alpha$'s marginally distributed as
$P_\alpha$.
It is called realizably monotone if
$X_\alpha\le X_\beta$ in $\mathcal{S}$ whenever $\alpha\le\beta$ in $\mathcal{A}$.
Such a system necessarily is stochastically monotone,
that is, it satisfies $P_\alpha\preceq P_\beta$ in stochastic ordering
whenever $\alpha \le \beta$.
It has been known exactly when these notions of monotonicity are equivalent
except for a certain subclass of acyclic posets, called Class~W.
In this paper we introduce inverse probability transforms
and synchronizing bijections recursively
when $\mathcal{S}$ is a poset of Class~W and $\mathcal{A}$ is synchronizable,
and validate monotonicity equivalence by
constructing $(X_\alpha: \alpha\in\mathcal{A})$ explicitly.
We also show that synchronizability is necessary for monotonicity
equivalence when $\mathcal{S}$ is in Class~W.

\medskip\par\noindent
{\em AMS\/} 2020 {\em subject classifications.\/}
Primary 60E05; 
secondary 06A06, 05C05, 05C38.



\end{abstract}

\section{Introduction}
\setcounter{equation}{0}\setcounter{figure}{0}

We express a finite partially ordered set (poset) by a calligraphic letter $\mathcal{S}$,
and the corresponding Hasse diagram by a plain uppercase letter $S$.
A subset $U$ of $\mathcal{S}$ is called an up-set if $x\in U$ and $x\le y$ in
$\mathcal{S}$ imply $y\in U$. A down-set $V$ of $\mathcal{S}$ is dually
defined, that is, it is an up-set of the dual $\mathcal{S}^*$
(which is equipped with the reversed ordering).
Throughout the present paper we view the Hasse diagram as an undirected graph,
and assume that it is connected.
We represent a graph $S = (V,E)$ with vertex set $V$ and edge set $E$,
but simply write $v\in S$ and $e\in S$
respectively for $v\in V$ and $e\in E$ if $v$ or $e$ is clearly identified as a vertex or an edge.
If $S$ is the induced subgraph of some supergraph $G$
and $v$ is an vertex of $G$ then
we define the subgraph $S-v$
of $S$ restricted on the vertex set $V\setminus\{v\}$,
and the supergraph $S+v$ by adding all the edges from $v$ to $S$ in $G$.
Similarly for any edge $e=\{a,b\}\in G$
we define the subgraph $S-e$ by $(V,E\setminus\{e\})$,
and the supergraph $S+e$ by $(V\cup\{a,b\},E\cup\{e\})$.

We say that $\mathcal{S}$ is a poset of Class~W
if (i) $S$ is a tree (i.e., an acyclic and connected graph)
and (ii) $\mathcal{S}$ has no induced Y-poset
[Figure~\ref{named.poset}(a)].
A poset $\mathcal{S}$ of Class~W can contain
$\mathrm{W}^{\star}$-posets or $\mathrm{W}_{\star}$-posets
[Figure~\ref{named.poset}(b)--(c)]
as induced subposets,
and it is called a poset of Class W$^{\star}$ if it contains only
$\mathrm{W}^{\star}$-posets.
We can dually characterize Class W$_{\star}$ by
$\mathrm{W}_{\star}$-posets.
Furthermore, we call it an up-down poset
if there is no induced $\mathrm{W}^{\star}$-poset nor $\mathrm{W}_{\star}$-poset
(cf. \cite{gansner1982}).
In Algorithm~\ref{rp.tree} we identify
a Hasse diagram $S$ of tree with
a rooted tree $(S,\tau)$ started from a leaf $\tau\in S$,
and generate subtrees $(S^{(\kappa)},u_1^{(\kappa)})$, $\kappa\in K$, recursively
along with rooted plane tree $K$ of indices.
In Definition~\ref{classw}
we formally characterize Class~W in terms of
subtrees $(S^{(\kappa)},u_1^{(\kappa)})$'s by Algorithm~\ref{rp.tree}.

\usetikzlibrary{positioning}
\tikzset{mynode/.style={draw,circle,inner sep=1.5pt,outer sep=0pt}}

\begin{figure}[h]\begin{center}
\begin{tabular}{ccc}
\begin{tikzpicture}
  \node [mynode] (w) at (0,0) {};
  \node [mynode] (z) at (0,-1) {};
  \node [mynode] (x) at (-1,-2) {};
  \node [mynode] (y) at ( 1,-2) {};
  \draw[thick]
  (w) -- (z)
  (z) -- (x)
  (z) -- (y);
  \node [mynode] (w) at (3,-2) {};
  \node [mynode] (z) at (3,-1) {};
  \node [mynode] (x) at (2, 0) {};
  \node [mynode] (y) at (4, 0) {};
  \draw[thick]
  (w) -- (z)
  (z) -- (x)
  (z) -- (y);
\end{tikzpicture}
&
\begin{tikzpicture}
  \node [mynode] (w) at ( 0,-1) {};
  \node [mynode] (x) at (-1, 0) {};
  \node [mynode] (y) at ( 0, 0) {};
  \node [mynode] (z) at ( 1, 0) {};
  \draw[thick]
  (w) -- (x)
  (w) -- (y)
  (w) -- (z);
\end{tikzpicture}
&
\begin{tikzpicture}
  \node [mynode] (w) at (0,0) {};
  \node [mynode] (x) at (-1,-1) {};
  \node [mynode] (y) at ( 0,-1) {};
  \node [mynode] (z) at ( 1,-1) {};
  \draw[thick]
  (w) -- (x)
  (w) -- (y)
  (w) -- (z);
\end{tikzpicture}
\\
(a) Y-posets
&
(b) $\mathrm{W}^{\star}$-poset
&
(c) $\mathrm{W}_{\star}$-poset
\end{tabular}
\caption{
  Two types of acyclic Hasse diagram.
}\label{named.poset}
\end{center}\end{figure}
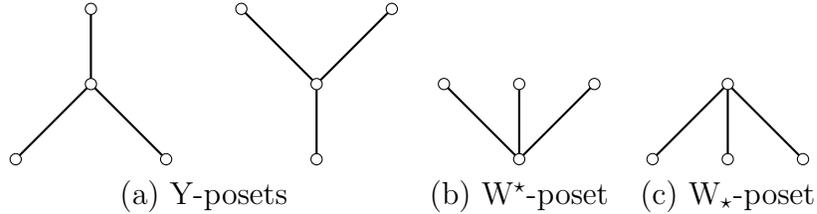

For probability measures $P$ and $P'$ on a poset $\mathcal{S}$,
we say that $P$ is stochastically smaller than $P'$, denoted by
$P\preceq P'$, if $P(U)\le P'(U)$ for every up-set $U$ of $\mathcal{S}$.
An important characterization of stochastic ordering was established
by Strassen~\cite{strassen} and further studied by Kamae et al.~\cite{kko}.
They showed that the stochastic ordering is necessary and sufficient
for the existence of a pair of $S$-valued random variables
$X$ and $X'$ distributed respectively as $P$ and $P'$
and satisfying $X\le X'$ in $\mathcal{S}$.
We can consider another poset $\mathcal{A}$,
and produce a system $(P_\gamma:\gamma\in\mathcal{A})$ of probability
measures on $\mathcal{S}$.
Then it is called \emph{stochastically monotone}
if $P_\alpha\preceq P_\beta$
whenever $\alpha < \beta$ in $\mathcal{A}$.
It is called \emph{realizably monotone} if there is a system
$(X_\gamma:\gamma\in\mathcal{A})$ of $\mathcal{S}$-valued random
variables marginally distributed as $P_\gamma$ for each $\gamma\in\mathcal{A}$
and satisfying $X_\alpha\le X_\beta$ in $\mathcal{S}$ whenever $\alpha<\beta$ in $\mathcal{A}$.
The two notions of monotonicity are not equivalent in general.
If they are equivalent for a pair
$(\mathcal{A},\mathcal{S})$ of particular posets
then we say that monotonicity equivalence holds.
In~\cite{fm2001,fm2002} we identified such pairs completely
except for the case when
$\mathcal{S}$ is in Class~W.

If $\mathcal{S}$ is an up-down poset
then the corresponding rooted tree $(S,\tau)$
is equipped with the linear ordering, denoted by $\le_{\tau}$,
with the maximum element $\tau$.
We can introduce the probability distribution function
$F_\gamma$ from $S$ to $[0,1]$ by
\begin{equation*}
  F_\gamma(x)=P_\gamma(\{z\in S: z\le_{\tau}x\})
\end{equation*}
and the inverse transform $X_\gamma$ from $[0,1)$ to $S$ by
\begin{equation*}
  X_\gamma(\omega)
  = \bigwedge\nolimits_{\tau}\{x\in S: \omega<F_\gamma(x)\}
\end{equation*}
which takes the minimum in terms of the linear ordering $\le_{\tau}$.
We set a probability space $(\lambda,\mathcal{B}([0,1)),[0,1))$
by the standard Lebesgue measure $\lambda$.
Then $X_\gamma$ can be viewed as an $S$-valued random variable,
and distributed as $P_\gamma$.
If $(P_\gamma:\gamma\in\mathcal{A})$ is stochastically monotone
then the system
$(X_\gamma:\gamma\in\mathcal{A})$ realizes the monotonicity;
thus, monotonicity equivalence holds when $\mathcal{S}$ is an up-down poset.
This assertion will be established
for the case (iv) of Theorem~\ref{me.bounded}.
In Section~\ref{me.sec} we further generalize an inverse transform
$X_{\mu,\gamma}$ when $\mathcal{S}$ is a poset of Class~W,
provided that there exists a distribution function $\mu$ on $K$
interlaced with $F_\gamma$.
Such distribution function $\mu$ is readily found when $\mathcal{A}$
has the minimum and the maximum element
by Lemma~\ref{interlace.lem},
resulting in the first sufficient condition
[i.e., the case (i) of Theorem~\ref{me.bounded}]
for monotonicity equivalence.

In Section~\ref{sync.sec} we extend sufficient conditions for monotonicity equivalence.
For a poset $\mathcal{A}$
we denote by $D_{\mathcal{A}}$ the collection of all the minimal
elements in $\mathcal{A}$.
We can introduce a connected graph
$G_{\mathcal{A}} = (D_{\mathcal{A}},E_{\mathcal{A}})$
with vertex set $D_{\mathcal{A}}$
and edge set $E_{\mathcal{A}}$
where $(\beta,\beta')\in E_{\mathcal{A}}$
if and only if there is some $\alpha\in\mathcal{A}$
such that $\beta,\beta'<\alpha$ in $\mathcal{A}$.
We call $G_{\mathcal{A}}$ the graph of \emph{interlaced relation}
in the sense of Remark~\ref{interlace.rem}(b).
A spanning tree $T$ of $G_{\mathcal{A}}$ is said to be
\emph{locally connected} if the subgraph of $T$ restricted
on the vertex set
\begin{equation}\label{d.set}
  D_{\mathcal{A}}(\alpha) = \{\beta\in D_{\mathcal{A}}:
  \beta\le\alpha \mbox{ in $\mathcal{A}$}\}
\end{equation}
is connected (i.e., it is a subtree) for each $\alpha\in\mathcal{A}$.
We call $\mathcal{A}$ \emph{synchronizable} for Class~$\mathrm{W}_{\star}$
if there is a locally connected spanning tree $T$ of $G_{\mathcal{A}}$.
Dually we can define a synchronizable poset for Class~$\mathrm{W}^{\star}$,
and we simply call a poset synchronizable
if it is synchronizable for both Class~$\mathrm{W}_{\star}$
and Class~$\mathrm{W}^{\star}$.

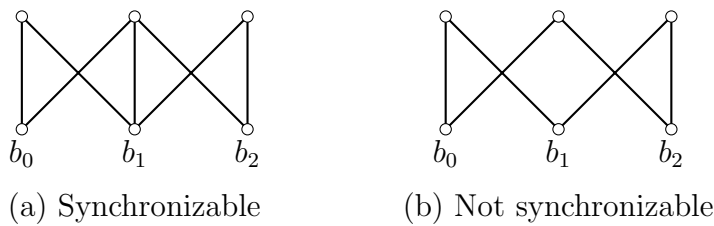
\begin{figure}[h]\begin{center}
\begin{tabular}{ccc}
\begin{tikzpicture}[xscale=1.5,yscale=1.5]
\draw[thick] (0,0) -- (0,1) -- (1,0) -- (2,1) -- (2,0) -- (1,1) -- (0,0);
\draw[thick] (1,0) -- (1,1);

\draw[fill=white] (0,0) circle [radius=0.05];
\node[below] at (0,0) {$b_0$};
\draw[fill=white] (1,0) circle [radius=0.05];
\node[below] at (1,0) {$b_1$};
\draw[fill=white] (2,0) circle [radius=0.05];
\node[below] at (2,0) {$b_2$};
\draw[fill=white] (0,1) circle [radius=0.05];
\draw[fill=white] (1,1) circle [radius=0.05];
\draw[fill=white] (2,1) circle [radius=0.05];
\end{tikzpicture}
&
\hspace{5ex}
&
\begin{tikzpicture}[xscale=1.5,yscale=1.5]
\draw[thick] (0,0) -- (0,1) -- (1,0) -- (2,1) -- (2,0) -- (1,1) -- (0,0);

\draw[fill=white] (0,0) circle [radius=0.05];
\node[below] at (0,0) {$b_0$};
\draw[fill=white] (1,0) circle [radius=0.05];
\node[below] at (1,0) {$b_1$};
\draw[fill=white] (2,0) circle [radius=0.05];
\node[below] at (2,0) {$b_2$};
\draw[fill=white] (0,1) circle [radius=0.05];
\draw[fill=white] (1,1) circle [radius=0.05];
\draw[fill=white] (2,1) circle [radius=0.05];
\end{tikzpicture}
\\
(a) Synchronizable
&
&
(b) Not synchronizable
\end{tabular}
\caption{
  Example and counter example of synchronizable poset
}\label{a.poset}
\end{center}\end{figure}

\begin{example}
Consider a poset $\mathcal{A}$ of Figure~\ref{a.poset}(a) or (b).
In either case we have the graph $G_{\mathcal{A}}$
with edge set
\begin{math}
  E_{\mathcal{A}}=\{\{b_0,b_1\},\{b_0,b_2\},\{b_1,b_2\}\} .
\end{math}
In (a) we can find a locally connected spanning tree $T$
with edges $\{b_0,b_1\}$ and $\{b_1,b_2\}$;
thus, $\mathcal{A}$ is synchronizable.
In (b) no spanning tree is locally connected, and therefore,
$\mathcal{A}$ is not synchronizable.
\end{example}

As examined in Section~\ref{me.proof.sec} and~\ref{rsb.proof.sec},
the inverse transform $X_{(\alpha,\beta),\gamma}$ is associated with
$(\alpha,\beta)\in D_{\mathcal{A}}\times D_{\mathcal{A}^*}$ and
$\gamma\in\mathcal{A}$.
In Section~\ref{rsb.sec} and~\ref{rsb.comp.sec}
we develop a well-defined measure-preserving bijection
$\Phi_{(\alpha,\beta)}$ when $\mathcal{A}$ is synchronizable.
In Section~\ref{rsb.proof.sec}
we show that an inverse transform $X_\gamma$ is well-defined
by $X_{(\alpha,\beta),\gamma}\circ\Phi_{(\alpha,\beta)}$,
and that the system $(X_\gamma:\gamma\in\mathcal{A})$ realizes the monotonicity;
thus, monotonicity equivalence holds when $\mathcal{A}$ is synchronizable.
In Section~\ref{wstar.sec} we conclude the proof of Theorem~\ref{me.sync}.
In Section~\ref{mi.sec}
we prove Proposition~\ref{main.claim},
claiming that the synchronizability of $\mathcal{A}$
is somewhat necessary for monotonicity equivalence
when $\mathcal{S}$ is in Class W.

\section{Monotonicity equivalence for Class W}\label{me.sec}
\setcounter{equation}{0}\setcounter{figure}{0}

In this section
we present a constructive proof for the following theorem
of monotonicity equivalence on a poset $\mathcal{S}$ of Class W.

\begin{theorem}\label{me.bounded}
Let $\mathcal{S}$ be a poset of Class W.
Then each of the following conditions sufficiently guarantees
monotonicity equivalence for $(\mathcal{A},\mathcal{S})$:
\begin{enumerate}
\renewcommand{\labelenumi}{(\roman{enumi})}
\item
$\mathcal{A}$ has the minimum $a_*$ and the maximum $a^*$;
\item
$\mathcal{A}$ has the minimum $a_*$
and $\mathcal{S}$ is in Class W$_{\star}$;
\item
$\mathcal{A}$ has the maximum $a^*$
and $\mathcal{S}$ is in Class W$^{\star}$;
\item
$\mathcal{S}$ is an up-down poset.
\end{enumerate}
\end{theorem}

\subsection{Rooted plane tree}
\label{rp.tree.sec}

Let $S$ be a tree, and let $\tau$ be a leaf of the tree $S$.
We can introduce a partial ordering for vertices $x,y\in S$,
denoted by $x\le_{\tau}y$, if the path from $\tau$ to $x$ on $S$
contains the path from $\tau$ to $y$.
We call the top element $\tau$ the \emph{root},
and the poset the \emph{rooted tree}, denoted by $(S,\tau)$.
In Algorithm~\ref{rp.tree}
we construct a collection of rooted subtrees
$(S^{(\kappa)},u_1^{(\kappa)})$, $\kappa\in K$,
indexed by a rooted tree $(K,1)$.

\begin{algorithm}\label{rp.tree}
We set $(S^{(1)},u_1^{(1)}) = (S,\tau)$ as the rooted tree at the root index
$1$ and start with $K=\{1\}$.
Provided the rooted subtree $(S^{(\kappa)},u_1^{(\kappa)})$ at a leaf
$\kappa\in K$, we can pursue children of it recursively.
\begin{enumerate}
\renewcommand{\labelenumi}{(\roman{enumi})}
\item
In the first step we identify the unique path
\begin{equation}\label{u.path}
  \mathbf{u}^{(\kappa)} = (u_1^{(\kappa)},u_2^{(\kappa)},\ldots,u_*^{(\kappa)})
\end{equation}
until it hits the vertex $u_*^{(1)}$ with no successor or more than one
successor in the rooted tree $(S^{(\kappa)},u_1^{(\kappa)})$.
If $u_*^{(1)}$ has no successor, there is no children and $\kappa\in K$
is labeled as ``terminated,''
indicated by $C(\kappa)=\varnothing$.
Otherwise, we can find the collection
$\{u_1^{(\sigma_1)},\ldots,u_1^{(\sigma_M)}\}$
of successors
[that is, the collection of vertices that $u_*^{(\kappa)}$ covers
in the rooted tree $(S,\tau)$]
indexed by
\begin{equation}\label{plane.order}
  C(\kappa) = \{\sigma_1,\ldots,\sigma_M\},
\end{equation}
and enumerate them
from the first $\sigma_1$ to the last $\sigma_M$;
thus, introducing a linear ordering on $C(\kappa)$.
\item
In the second step
we extend $K$ to $K\cup C(\kappa)$,
where each successor $u_1^{(\sigma_i)}$ is identified with a leaf of a subtree component
$S^{(\sigma_i)}$ of the forest $S^{(\kappa)}-\mathbf{u}^{(\kappa)}$,
and each rooted subtree $(S^{(\sigma_i)},u_1^{(\sigma_i)})$
becomes a child of $(S^{(\kappa)},u_1^{(\kappa)})$.
\end{enumerate}
We identify all the rooted subtree $(S^{(\kappa)},u_1^{(\kappa)})$
by repeating these steps recursively until
all the leaves $\kappa\in K$ are labeled as ``terminated.''
\end{algorithm}

In Algorithm~\ref{rp.tree}
the collection $K$ of indices $\kappa$'s
starts from the root $1$,
and each successor $\sigma_i$ is associated with
linearly ordered set $C(\kappa)$ of \eqref{plane.order}.
Then the rooted tree $(K,1)$ is viewed as a \emph{rooted plane tree} (cf.~\cite{bender2013}).
By $\le_{1}$ we denote the partial ordering introduced by the rooted tree $(K,1)$.
We have a natural linear extension of the rooted tree $(K,1)$ by
lexicographic ordering, denoted by $\le_{\mathrm{lex}}$.
When a pair $\{\kappa,\kappa'\}$ of $K$ is not comparable in $\le_{1}$
we can find the least upper bound $\hat{\kappa}$ of the pair
with $\kappa\le_{1}\sigma_i$ and $\kappa'\le_{1}\sigma_j$
for some distinct $\sigma_i,\sigma_j\in C(\hat{\kappa})$,
and therefore, we can introduce $\kappa<_{\mathrm{lex}}\kappa'$
if $\sigma_i<\sigma_j$ in the linearly ordered $C(\hat{\kappa})$.

In Definition~\ref{classw}
we consider a poset $\mathcal{S}$ whose Hasse diagram $S$ is a tree,
and arbitrarily select a rooted tree $(S,\tau)$ with leaf $\tau\in S$.
Then Class W is equivalently characterized by the tails $u_*^{(\kappa)}$'s
of (\ref{u.path}) generated by Algorithm~\ref{rp.tree}.

\begin{definition}\label{classw}
The poset $\mathcal{S}$ is in Class W
if and only if the tail $u_*^{(\kappa)}$ of (\ref{u.path}) is minimal or maximal in
$\mathcal{S}$.
Assuming $C(1)\neq\varnothing$,
the poset $\mathcal{S}$ of Class W belongs to Class W$_{\star}$
if the tail $u_*^{(\kappa)}$ is always maximal in $\mathcal{S}$ whenever $C(\kappa)\neq\varnothing$.
Class W$^{\star}$ is dually characterized by the minimality of
$u_*^{(\kappa)}$ whenever $C(\kappa)\neq\varnothing$.
If $C(1)=\varnothing$ then $\mathcal{S}$ becomes an up-down poset.
\end{definition}

\subsection{Recursive inverse transforms}

Here we assume that $\mathcal{S}$ is a poset of Class W,
and continue the setting of Definition~\ref{classw}
where the rooted tree $(S,\tau)$ is selected
and the rooted tree $(K,1)$ is generated by Algorithm~\ref{rp.tree}.
For each $x\in S$ we can define
the closed section $(\leftarrow,x] = \{z\in S: z\le_{\tau}x\}$
and the open section $(\leftarrow,x) = \{z\in S: z<_{\tau}x\}$
by the rooted tree $(S,\tau)$.
We can observe that each of these sections is either a down-set or an up-set
of $\mathcal{S}$.
For a measure $P$ (i.e., a nonnegative additive set function) on $S$
we can introduce distribution functions $F(x)$ and $F(x-)$ on $S$ respectively by
\begin{equation*}
  F(x) = P((\leftarrow,x])
    \mbox{ and }
  F(x-) = P((\leftarrow,x)) = F(x) - P(\{x\})
\end{equation*}
for $x\in S$.
$F$ is said to be smaller than $F'$ on $\mathcal{S}$, denoted by $F\preceq F'$,
if
\begin{equation*}
  \begin{cases}
    F(x) \ge F'(x)
    & \mbox{if $(\leftarrow,x]$ is a down-set of $\mathcal{S}$; } \\
    F(x) \le F'(x)
    & \mbox{if $(\leftarrow,x]$ is an up-set of $\mathcal{S}$ }
  \end{cases}
\end{equation*}
for each $x\in S$.
Note that the comparability between $F$ and $F'$ implies $F(\tau)=F'(\tau)$.
When $F$ and $F'$ correspond to the respective probability measures $P$ and $P'$,
$F\preceq F'$ is equivalent to the stochastic order $P\preceq P'$.

Similarly we can define a distribution function $\mu(\kappa)$ on $K$ if it satisfies
\begin{equation*}
  \mu(\kappa)
  \ge\mu(\kappa-) = \sum_{\sigma_i\in C(\kappa)}\mu(\sigma_i)
  \mbox{ for each $\kappa\in K$, }
\end{equation*}
where the summation above is zero when $C(\kappa)=\varnothing$.
We say that $\mu$ is interlaced with a distribution function $F$ on $S$
if $\mu$ and $F$ satisfy $\mu(1)=F(\tau)$ and
\begin{equation}\label{interlace}
  \mu(\kappa-)\le F(u_*^{(\kappa)})\le F(u_1^{(\kappa)})\le\mu(\kappa)
\end{equation}
for each path $\mathbf{u}^{(\kappa)}$ of \eqref{u.path}
constructed by Algorithm~\ref{rp.tree}.
For each vertex $\sigma_j$ of the linearly ordered set
$C(\kappa) = \{\sigma_1,\ldots,\sigma_M\}$ of \eqref{plane.order}
we can introduce $\mu^{(\kappa)}\lceil\sigma_j\rceil$ and
$\mu^{(\kappa)}\lfloor\sigma_j\rfloor$ by
\begin{align}\label{lex.ceil}
  \mu^{(\kappa)}\lceil\sigma_j\rceil &= \sum_{i=1}^j \mu(\sigma_i);
  \\ \label{lex.floor}
  \mu^{(\kappa)}\lfloor\sigma_j\rfloor &= \sum_{i=1}^{j-1} \mu(\sigma_i)
\end{align}

In what follows we consider a system $(F_\alpha:\alpha\in A)$ of
distribution functions on $S$ satisfying $F_\alpha(\tau) = c > 0$ for each
$\alpha\in A$,
and assume that there exists an interlaced distribution function $\mu$ on $K$
with all $F_\alpha$'s.
We set $\hat{S}^{(1)} = S$, and extend
$S^{(\kappa)}$ to $\hat{S}^{(\kappa)} = S^{(\kappa)}+u_*^{(\hat{\kappa})}$
if $\hat{\kappa}$ is the parent of $\kappa\in K\setminus\{1\}$.
Similarly we extend the path $\mathbf{u}^{(\kappa)}$
by setting
$\hat{\mathbf{u}}^{(1)} =\mathbf{u}^{(1)}$ and
\begin{equation}\label{u.extend}
  \hat{\mathbf{u}}^{(\kappa)} =\mathbf{u}^{(\kappa)}+u_*^{(\hat{\kappa})}
  =(u_*^{(\hat{\kappa})},u_1^{(\kappa)},\ldots,u_*^{(\kappa)})
\end{equation}
to include the tail element
$u_*^{(\hat{\kappa})}$ of the parent path $\mathbf{u}^{(\hat{\kappa})}$
if $\kappa\in K\setminus\{1\}$.
Then we can introduce a distribution function $F_\alpha^{(\kappa)}$ on $\hat{S}^{(\kappa)}$ by setting
$F_\alpha^{(1)}=F_\alpha$ and
\begin{equation}\label{f.recursion}
  F_\alpha^{(\kappa)}(x) = \begin{cases}
    F_\alpha(x) & \mbox{if $x\in S^{(\kappa)}$;} \\
    \mu(\kappa) & \mbox{if $x=u_*^{(\hat{\kappa})}$}
  \end{cases}
\end{equation}
for $\kappa\in K\setminus\{1\}$.
In Algorithm~\ref{map.alg}
for each $\alpha\in A$
we can recursively construct an inverse transform
$X_{\mu,\alpha}^{(\kappa)}$
from $[0,\mu(\kappa))$ to $\hat{S}^{(\kappa)}$,
and simply write $X_{\mu,\alpha}$ for the inverse transform
$X_{\mu,\alpha}^{(1)}$ from $[0,\mu(1))$ to $S$.

\begin{algorithm}\label{map.alg}
If $C(\kappa)=\varnothing$ then $\mu(\kappa-)=0$.
Otherwise, we have the linearly ordered set $C(\kappa)$ of \eqref{plane.order}
for which the inverse transform $X_{\mu,\alpha}^{(\sigma_i)}$
from $[0,\mu(\sigma_i))$ to $\hat{S}^{(\sigma_i)}$ can be constructed by recursion.
We set for $\omega\in [0,\mu(\kappa-))$
\begin{equation}\label{map.recursion}
  X_{\mu,\alpha}^{(\kappa)}(\omega) = X_{\mu,\alpha}^{(\sigma_i)}(\omega-\mu^{(\kappa)}\lfloor\sigma_i\rfloor)
  \mbox{ if $\mu^{(\kappa)}\lfloor\sigma_i\rfloor\le\omega<\mu^{(\kappa)}\lceil\sigma_i\rceil$}
\end{equation}
for some $\sigma_i\in C(\kappa)$.
We extend the map $X_{\mu,\alpha}^{(\kappa)}$ from $[0,\mu(\kappa))$
to $\hat{S}^{(\kappa)}$ by setting for $\omega\in [\mu(\kappa-),\mu(\kappa))$,
\begin{equation}\label{map.inverse}
  X_{\mu,\alpha}^{(\kappa)}(\omega)
  = \bigwedge\nolimits_{\tau}\{x\in\hat{\mathbf{u}}^{(\kappa)}: \omega<F_\alpha^{(\kappa)}(x)\}
\end{equation}
which returns the minimum in the linearly ordered set
$(\hat{\mathbf{u}}^{(\kappa)},\le_{\tau})$ of \eqref{u.extend}.
\end{algorithm}

\begin{remark}\label{map.rem}
In the construction of $X_{\mu,\alpha}^{(\kappa)}$ in Algorithm~\ref{map.alg}
we can observe that $X_{\mu,\alpha}^{(\kappa)}(\omega) = u_*^{(\kappa)}$
if and only if
\begin{equation}\label{map.i}
  \omega\in
  I_{\mu,\alpha}^{(\kappa)} =
  \left[\mu(\kappa-),F_\alpha(u_*^{(\kappa)})\right)
    \cup
    \bigcup_{\sigma_i\in C(\kappa)}\left\{
    \mu^{(\kappa)}\lfloor\sigma_i\rfloor +
       \left[F_\alpha(u_1^{(\sigma_i)}),\mu(\sigma_i)\right)
         \right\}
\end{equation}
where the union over $C(\kappa)$ is the empty set if $C(\kappa)=\varnothing$.
\end{remark}

\begin{proposition}\label{map.prop}
The recursive inverse transform $X_{\mu,\alpha}^{(\kappa)}$ of Algorithm~\ref{map.alg} realizes
$F_\alpha^{(\kappa)}$ in the sense that
\begin{equation}\label{map.realize}
  \lambda\left(\left\{\omega:
  X_{\mu,\alpha}^{(\kappa)}(\omega)\in(\leftarrow,x]
  \right\}\right)
  = F_\alpha^{(\kappa)}(x),
  \quad x\in\hat{S}^{(\kappa)},
\end{equation}
where $\lambda$ is the Lebesgue measure on $[0,\mu(\kappa))$.
Furthermore, it is monotone so that
\begin{equation}\label{map.monotone}
  X_{\mu,\alpha}^{(\kappa)}(\omega)\le X_{\mu,\beta}^{(\kappa)}(\omega)
  \mbox{ in $\mathcal{S}$ for any $\omega\in [0,\mu(\kappa))$}
\end{equation}
whenever $F_\alpha^{(\kappa)}\preceq F_\beta^{(\kappa)}$.
\end{proposition}

\begin{proof}
The realization \eqref{map.realize} and the monotonicity \eqref{map.monotone}
can be shown by induction on the rooted plane tree $(K,1)$.
First consider the base case when $C(\kappa)=\varnothing$.
Then $\hat{S}^{(\kappa)}$
becomes the linearly ordered set
\begin{equation}\label{u.hat}
  \hat{\mathbf{u}}^{(\kappa)} = (x_1,\ldots,x_N),
\end{equation}
of \eqref{u.extend} with root $x_1=u_*^{(\hat{\kappa})}$
if the parent $\hat{\kappa}$ of $\kappa$ exists;
otherwise, $\kappa = 1$ and $x_1=\tau$.
By (\ref{map.inverse}) we obtain
for $\omega\in [0,\mu(\kappa))$
\begin{equation}\label{map.interval}
  X_{\mu,\alpha}^{(\kappa)}(\omega)=x_i
  \mbox{ if and only if }
  F_\alpha^{(\kappa)}(x_i-)\le\omega<F_\alpha^{(\kappa)}(x_i)
\end{equation}
which immediately implies (\ref{map.realize}).
Suppose that $F_\alpha^{(\kappa)}\preceq F_\beta^{(\kappa)}$,
and that $x_j=X_{\mu,\beta}^{(\kappa)}(\omega)$ for an arbitrarily fixed
$\omega\in [0,\mu(\kappa))$.
If $i\ge j+1$ then
\begin{equation*}
  F_\beta^{(\kappa)}(x_k)\le F_\beta^{(\kappa)}(x_j-)
  \le\omega<F_\alpha^{(\kappa)}(x_i)\le F_\alpha^{(\kappa)}(x_k)
\end{equation*}
for $k=j+1,\ldots,i$;
thus, $(\leftarrow,x_k]=\{x_k,\ldots,x_N\}$ must be
a down-set of $\mathcal{S}$,
and consequently,
$X_{\mu,\alpha}^{(\kappa)}(\omega)=x_i<\cdots<x_{j+1}<x_j=X_{\mu,\beta}^{(\kappa)}(\omega)$ in $\mathcal{S}$.
Similarly if $j\ge i+1$ then
$(\leftarrow,x_k]$ is an up-set of $\mathcal{S}$
for $k=i+1,\ldots,j$, and therefore,
$X_{\mu,\beta}^{(\kappa)}(\omega)=x_j>\cdots>x_{i+1}>x_i=X_{\mu,\alpha}^{(\kappa)}(\omega)$ in $\mathcal{S}$.
In either case the monotonicity (\ref{map.monotone}) has been verified.

Suppose that $C(\kappa)\neq\varnothing$
and that the induction hypothesis holds for each $\sigma_i\in C(\kappa)$.
For $\omega\in [\mu(\kappa-),\mu(\kappa))$
we can express \eqref{map.inverse} equivalently by \eqref{map.interval} on
$\hat{\mathbf{u}}^{(\kappa)}$ of \eqref{u.hat}.
By the recursive construction of \eqref{map.recursion}
the induction hypothesis implies \eqref{map.realize}
for $x\in\hat{S}^{(\kappa)}$
as well as the monotonicity of \eqref{map.monotone}
for $\omega\in [0,\mu(\kappa-))$.
As it was demonstrated for the base case,
the construction of \eqref{map.interval}
ensures (\ref{map.monotone}) for $\omega\in [\mu(\kappa-),\mu(\kappa))$
whenever $F_\alpha^{(\kappa)}\preceq F_\beta^{(\kappa)}$.
Thus, we have completed the proof.
\end{proof}

Let $\mathcal{A}$ be a poset,
and let $(F_\alpha:\alpha\in\mathcal{A})$ be a stochastically monotone system
of probability distribution functions on $\mathcal{S}$
[i.e., $F_\alpha(\tau)=1$].
By Algorithm~\ref{map.alg} we can construct an inverse transform $X_{\mu,\alpha}=X_{\mu,\alpha}^{(1)}$
from $[0,1)$ to $\mathcal{S}$,
and view it as $\mathcal{S}$-valued random variable on the sample space
$\Omega=[0,1)$ with probability measure $\lambda$.
Thus, the following corollary is an immediate consequence of
Proposition~\ref{map.prop}.

\begin{corollary}\label{map.cor}
If there exists an interlaced distribution function $\mu$ on $K$
with all $F_\alpha$'s
then the stochastically monotone system $(F_\alpha:\alpha\in\mathcal{A})$ is
realizably monotone.
\end{corollary}

\subsection{Proof of Theorem~\ref{me.bounded}}
\label{me.proof.sec}

To complete the proof it suffices to find an interlaced distribution function
of Corollary~\ref{map.cor}
under each of the four conditions (i)--(iv) of Theorem~\ref{me.bounded}.

\begin{lemma}\label{interlace.lem}
Let $\mathcal{S}$ be a poset of Class W, and
let $F_\alpha,F_\beta,F$ be distribution functions on $\mathcal{S}$.
\begin{enumerate}
\renewcommand{\labelenumi}{(\roman{enumi})}
\item
If $F_\alpha\preceq F\preceq F_\beta$ then
\begin{equation}\label{ab.dist}  
  \mu_{(\alpha,\beta)}(\kappa)=F_{\alpha}\vee F_{\beta}(u_1^{(\kappa)}),
  \quad\kappa\in K,
\end{equation}
is interlaced with $F$.
\item
If $F_\alpha\preceq F$ and $\mathcal{S}$ is in Class $\mathrm{W}_{\star}$ then
\begin{equation}\label{a.dist}  
  \mu_\alpha(\kappa)=F_{\alpha}(u_1^{(\kappa)}),
  \quad\kappa\in K,
\end{equation}
is interlaced with $F$.
\item
If $F\preceq F_\alpha$ and $\mathcal{S}$ is in Class $\mathrm{W}^{\star}$
then $\mu_\alpha$ of \eqref{a.dist} is interlaced with $F$.
\end{enumerate}
\end{lemma}

Note in Lemma~\ref{interlace.lem} and the rest of paper that
we simply write $F_{\alpha}\vee F_{\beta}(x)$ and
$F_{\alpha}\wedge F_{\beta}(x)$ respectively
for $\max\{F_{\alpha}(x),F_{\beta}(x)\}$
and $\min\{F_{\alpha}(x),F_{\beta}(x)\}$.

\begin{proof}[Proof of Lemma~\ref{interlace.lem}]
Consider the case (i).
If $C(\kappa)=\varnothing$ then $\mu_{(\alpha,\beta)}$ of \eqref{ab.dist} satisfies
(\ref{interlace}).
Otherwise, for every $\sigma_i\in C(\kappa)$
the closed section
$(\leftarrow,u_1^{(\sigma_i)}]$ is a down-set of $\mathcal{S}$
if $(\leftarrow,u_*^{(\kappa)}]$ is an up-set;
$(\leftarrow,u_1^{(\sigma_i)}]$ is an up-set of $\mathcal{S}$
if $(\leftarrow,u_*^{(\kappa)}]$ is a down-set.
Thus, we can observe that
\begin{equation*}
  \mu_{(\alpha,\beta)}(\kappa-)
  = \sum_{\sigma_i\in C(\kappa)}\mu_{(\alpha,\beta)}(\sigma_i)
  \le F_\alpha\vee F_\beta(u_*^{(\kappa)}-)
  \le F_\alpha\wedge F_\beta(u_*^{(\kappa)})
  \le F(u_*^{(\kappa)})
\end{equation*}
and conclude that $\mu_{(\alpha,\beta)}$ satisfies (\ref{interlace}).

Consider the case (ii).
Then $(\leftarrow,u_*^{(\kappa)}]$ is an up-set of $\mathcal{S}$
if $C(\kappa)\neq\varnothing$,
and $(\leftarrow,u_1^{(\kappa)}]$ is a down-set of $\mathcal{S}$.
Hence, (\ref{interlace}) is implied by
\begin{equation*}
  \mu_\alpha(\kappa-) = \sum_{\sigma_i\in C(\kappa)}F_\alpha(u_1^{(\sigma_i)})
  \le F_\alpha(u_*^{(\kappa)})
  \le F(u_*^{(\kappa)})
  \le F(u_1^{(\kappa)})
  \le F_\alpha(u_1^{(\kappa)}).
\end{equation*}
The assertion of (iii) is dually verified.
\end{proof}

Under Theorem~\ref{me.bounded}(i)
we find $\mu_{(a_*,a^*)}$ of \eqref{ab.dist} desired for Corollary~\ref{map.cor}.
In the case of Theorem~\ref{me.bounded}(ii) and (iii)
we can choose $\mu_\alpha$ of \eqref{a.dist} with $\alpha=a_*$ and $\alpha=a^*$ respectively.
For Theorem~\ref{me.bounded}(iv)
we find $K=\{1\}$ and
$\mu(1)=1$ trivially interlaced with any $F_\alpha$.
In each case of Theorem~\ref{me.bounded}
we have established monotonicity equivalence by Corollary~\ref{map.cor}.

\section{Monotonicity equivalence by synchronization}\label{sync.sec}
\setcounter{equation}{0}\setcounter{figure}{0}

In this section we investigate how the notion of synchronizability
extends the sufficient conditions of Theorem~\ref{me.bounded}
for monotonicity equivalence,
and establish the following theorem.

\begin{theorem}\label{me.sync}
Let $\mathcal{S}$ be a poset of Class W.
Then each of the following conditions sufficiently
guarantees monotonicity equivalence for $(\mathcal{A},\mathcal{S})$:
\begin{enumerate}
\renewcommand{\labelenumi}{(\roman{enumi})}
\item
$\mathcal{A}$ is a synchronizable poset;
\item
$\mathcal{S}$ is in Class W$_{\star}$ and
$\mathcal{A}$ is a synchronizable poset for Class W$_{\star}$;
\item
$\mathcal{S}$ is in Class W$^{\star}$ and
$\mathcal{A}$ is a synchronizable poset for Class W$^{\star}$.
\end{enumerate}
\end{theorem}

\subsection{Recursive synchronizing bijections}
\label{rsb.sec}

Let $(S,\tau)$ be a rooted tree of a poset $\mathcal{S}$ of Class W,
and let $(K,1)$ be a rooted plane tree
constructed by Algorithm~\ref{rp.tree}.
A pair $(\mu,\mu')$ of distribution functions on $K$
is said to be \emph{mutually interlaced} if
for each $\kappa\in K$
\begin{equation*}
  \textstyle
  \mu\vee\mu'(\kappa-)
  \le\mu\wedge\mu'(\kappa)
\end{equation*}
where we simply write
$\mu\vee\mu'(\kappa-)$
and $\mu\wedge\mu'(\kappa)$
respectively for
$\max\{\mu(\kappa-),\mu'(\kappa-)\}$
and
$\min\{\mu(\kappa),\mu'(\kappa)\}$.
It should be noted particularly that
if both $\mu$ and $\mu'$ are interlaced with some distribution function $F$ on
$\mathcal{S}$ then the pair $(\mu,\mu')$ is mutually interlaced.

\begin{remark}\label{interlace.rem}
Consider a stochastically monotone system
$(F_\alpha:\alpha\in\mathcal{A})$ of probability distribution functions on $\mathcal{S}$.
By Lemma~\ref{interlace.lem} we find a mutually interlaced pair of the following cases:
\begin{enumerate}
\renewcommand{\labelenumi}{(\alph{enumi})}
\item
We can construct $\mu_{(\alpha,\beta)}$ and $\mu_{(\alpha',\beta')}$ by \eqref{ab.dist}
respectively with $(F_\alpha,F_\beta)$ and $(F_{\alpha'},F_{\beta'})$
if $\alpha,\alpha'\le\gamma\le\beta,\beta'$ for some $\gamma\in\mathcal{A}$.
\item
We can construct $\mu_{\alpha}$ and $\mu_{\alpha'}$ by \eqref{a.dist}
respectively with $F_\alpha$ and $F_{\alpha'}$
if $\mathcal{S}$ is in Class~W$_{\star}$
and $\alpha,\alpha'\le\gamma$ with some $\gamma\in\mathcal{A}$;
or
if $\mathcal{S}$ is in Class~W$^{\star}$
and $\gamma\le\alpha,\alpha'$ with some $\gamma\in\mathcal{A}$.
\end{enumerate}
\end{remark}

Assuming that $I$ and $I'$ are Borel measurable subsets of $\mathbb{R}$
of equal length [i.e., $\lambda(I) = \lambda(I')$],
a Borel-measurable bijection $\Phi$ from $I$ to $I'$ is said to be \emph{measure-preserving}
if $\lambda(\Phi^{-1}(B)) = \lambda(B)$ for any Borel measurable subset $B$ of $I'$.
Provided an interlaced pair $(\mu,\mu')$,
we can construct a measure-preserving bijection $\Phi_{\mu,\mu'}^{(\kappa)}$
from $[0,\mu\wedge\mu'(\kappa))$ to itself
recursively by Algorithm~\ref{rst.cons}.
We call it
a \emph{recursive synchronizing bijection} (RSB) of $(\mu,\mu')$,
and simply write $\Phi_{\mu,\mu'}$ for an RSB $\Phi_{\mu,\mu'}^{(1)}$.

\begin{algorithm}\label{rst.cons}
We construct a measure-preserving bijection $\Phi_{\mu,\mu'}^{(\kappa)}$ recursively
on the rooted plane tree $(K,1)$.
If $C(\kappa)=\{\sigma_1,\ldots,\sigma_M\}\neq\varnothing$
then we have an RSB $\Phi_{\mu,\mu'}^{(\sigma_i)}$
from $[0,\mu\wedge\mu'(\sigma_i))$ to itself by recursion
for $i=1,\ldots,M$.
By using $\mu^{(\kappa)}\lfloor\sigma_i\rfloor$ of (\ref{lex.floor})
we can construct
\begin{equation}\label{rst.sigma}
  \Phi_{\mu,\mu'}^{(\kappa)}(\omega) =
  \Phi_{\mu,\mu'}^{(\sigma_i)}(\omega-\mu^{(\kappa)}\lfloor\sigma_i\rfloor)
  + \mu'^{(\kappa)}\lfloor\sigma_i\rfloor
\end{equation}
from
$\mu^{(\kappa)}\lfloor\sigma_i\rfloor +
[0,\mu\wedge\mu'(\sigma_i))$
to
$\mu'^{(\kappa)}\lfloor\sigma_i\rfloor +
[0,\mu\wedge\mu'(\sigma_i))$.
Furthermore, we set the identity map $\Phi_{\mu,\mu'}^{(\kappa)}$ on
\begin{equation*}
  \left[\mu\vee\mu'(\kappa-),\mu\wedge\mu'(\kappa)\right),
\end{equation*}
and extend a measure-preserving bijection of
$\Phi_{\mu,\mu'}^{(\kappa)}$ from
\begin{equation}\label{j.domain}
  J_{\mu,\mu'}^{(\kappa)} =
  \left[\mu(\kappa-),\mu\vee\mu'(\kappa-)\right)
    \cup
    \bigcup_{\sigma_i\in C(\kappa)}
    \left\{
    \mu^{(\kappa)}\lfloor\sigma_i\rfloor +
       \left[\mu\wedge\mu'(\sigma_i),\mu(\sigma_i)\right)
         \right\}
\end{equation}
to
\begin{equation}\label{j.range}
  {J'}_{\mu,\mu'}^{(\kappa)} =
  \left[\mu'(\kappa-),\mu\vee\mu'(\kappa-)\right)
    \cup
    \bigcup_{\sigma_i\in C(\kappa)}
    \left\{
    \mu'^{(\kappa)}\lfloor\sigma_i\rfloor +
       \left[\mu\wedge\mu'(\sigma_i),\mu'(\sigma_i)\right)
         \right\}
\end{equation}
of the equal length
\begin{equation*}
  \lambda(J_{\mu,\mu'}^{(\kappa)}) = \lambda(J_{\mu,\mu'}'^{(\kappa)})
  =\mu\vee\mu'(\kappa-)
  -\sum_{\sigma_i\in C(\kappa)}\mu\wedge\mu'(\sigma_i).
\end{equation*}
\end{algorithm}

\begin{remark}\label{rst.rem}
(a) Algorithm~\ref{rst.cons} starts with
the base case when $C(\kappa)=\varnothing$.
It should be noted for the base case
that
${J}_{\mu,\mu'}^{(\kappa)} = {J'}_{\mu,\mu'}^{(\kappa)} = \varnothing$,
and therefore, that
$\Phi_{\mu,\mu'}^{(\kappa)}$
is the identity map from
$[0,\mu\wedge\mu'(\kappa))$ to itself.
(b)
The inverse $\Phi_{\mu,\mu'}^{-1}$ becomes an RSB
of $(\mu',\mu)$ by Algorithm~\ref{rst.cons}.
\end{remark}

Assuming that the underlying distribution functions $F_\alpha$ and $F_\beta$
are indexed by $\alpha$ and $\beta$,
we refer $\mu_{(\alpha,\beta)}$ and $\mu_{\alpha}$ respectively to
\eqref{ab.dist} and \eqref{a.dist}.

\begin{lemma}\label{abcd.pair}
Let $F_a,F_b,F_c,F_d$ be distribution functions on $\mathcal{S}$.
If $F_a,F_b\preceq F_c,F_d$
then $\mu_{(a,c)}$ and $\mu_{(b,d)}$ are mutually interlaced.
\end{lemma}

\begin{proof}
Suppose that $(\leftarrow,u_*^{(\kappa)}]$ is an up-set of $\mathcal{S}$.
Then we find $(\leftarrow,u_*^{(\kappa)})$ a down-set,
and therefore, we can observe that
\begin{equation}\label{abcd.k-}
  \mu_{(\alpha,\beta)}(\kappa-) = \mu_\alpha(\kappa-)
\end{equation}
for $\alpha\in\{a,b\}$ and $\beta\in\{c,d\}$.
Thus, we obtain
\begin{align}
  \nonumber
  &\mu_{(a,c)}\vee\mu_{(b,d)}(\kappa-)
  = \mu_a\vee\mu_b(\kappa-)
  \le F_c\wedge F_d(u_*^{(\kappa)})
  \\ & \hspace{6ex} \label{abcd.k}
  \le \mu_{(a,c)}\wedge\mu_{(b,d)}(\kappa)
  = \begin{cases}
    \mu_a\wedge\mu_b(\kappa)
    & \mbox{if $(\leftarrow,u_1^{(\kappa)}]$ is a down-set;} \\
    \mu_c\wedge\mu_d(\kappa)
    & \mbox{if $(\leftarrow,u_1^{(\kappa)}]$ is an up-set,}
  \end{cases}
\end{align}
The case for $(\leftarrow,u_*^{(\kappa)}]$ being a down-set can be similarly verified.
\end{proof}

We note in Lemma~\ref{abcd.pair} that
the interlaced pair $(\mu_{(a,c)},\mu_{(b,d)})$
does not require the existence of $F$
such that $F_a,F_b\preceq F\preceq F_c,F_d$.
Furthermore, we obtain the well-definedness of RSB's in the sense of the
following lemma.

\begin{lemma}\label{abcd.rst}
Let $F_a,F_b,F_c,F_d$ be distribution functions on $\mathcal{S}$ satisfying
$F_a,F_b\preceq F_c,F_d$,
and let $\Phi_{(\alpha,\beta),(\alpha',\beta')}^{(\kappa)}$ be the RSB constructed by
Algorithm~\ref{rst.cons}
for an interlaced pair $(\mu_{(\alpha,\beta)},\mu_{(\alpha',\beta')})$
with $\alpha,\alpha'\in\{a,b\}$ and $\beta,\beta'\in\{c,d\}$.
If the bijection of Algorithm~\ref{rst.cons}
is uniquely assigned for each pair
$(J_{\mu,\mu'}^{(\kappa)},J_{\mu,\mu'}'^{(\kappa)})$ of \eqref{j.domain} and \eqref{j.range}
then we have
\begin{equation*}
  \Phi_{(a,c),(b,d)}^{(\kappa)}
  = \Phi_{(b,c),(b,d)}^{(\kappa)}\circ\Phi_{(a,c),(b,c)}^{(\kappa)}
  = \Phi_{(a,d),(b,d)}^{(\kappa)}\circ\Phi_{(a,c),(a,d)}^{(\kappa)}
\end{equation*}
from $[0,\mu_{(a,c)}\wedge\mu_{(b,d)}(\kappa))$ to itself.
\end{lemma}

\begin{proof}
For each $\kappa\in K$
we can show
\begin{equation}\label{abcd.comp}
  \Phi_{(a,c),(b,d)}^{(\kappa)}
  = \Phi_{(b,c),(b,d)}^{(\kappa)}\circ\Phi_{(a,c),(b,c)}^{(\kappa)}
\end{equation}
mapping from
\begin{equation*}
  J_{\mu_{(a,c)},\mu_{(b,d)}}^{(\kappa)}\cup
  [\mu_{(a,c)}\vee\mu_{(b,d)}(\kappa-),\mu_{(a,c)}\wedge\mu_{(b,d)}(\kappa))
\end{equation*}
to
\begin{equation*}
  {J'}_{\mu_{(a,c)},\mu_{(b,d)}}^{(\kappa)}\cup
  [\mu_{(a,c)}\vee\mu_{(b,d)}(\kappa-),\mu_{(a,c)}\wedge\mu_{(b,d)}(\kappa))
\end{equation*}
when $(\leftarrow,u_*^{(\kappa)}]$ is an up-set of $\mathcal{S}$.
We can similarly handle the case when $(\leftarrow,u_*^{(\kappa)}]$ is a down-set.
As in the proof of Lemma~\ref{abcd.pair},
we can observe \eqref{abcd.k-} and \eqref{abcd.k}.
The map $\Phi_{(a,c),(b,c)}^{(\kappa)}$ is identity on
\begin{equation*}
  [\mu_a\vee\mu_b(\kappa-),\mu_{(a,c)}\wedge\mu_{(b,c)}(\kappa)),
\end{equation*}
and it is bijective from
\begin{equation*}
  J_{\mu_{(a,c)},\mu_{(b,c)}}^{(\kappa)} =
  [\mu_a(\kappa-),\mu_a\vee\mu_b(\kappa-))
    \cup\bigcup_{\sigma_i\in C(\kappa)}
    \left\{
    \mu_a^{(\kappa)}\lfloor\sigma_i\rfloor +
       \left[\mu_a\wedge\mu_b(\sigma_i),\mu_a(\sigma_i)\right)
         \right\}
\end{equation*}
to
\begin{equation*}
  J_{\mu_{(a,c)},\mu_{(b,c)}}'^{(\kappa)} =
  [\mu_b(\kappa-),\mu_a\vee\mu_b(\kappa-))
    \cup\bigcup_{\sigma_i\in C(\kappa)}
    \left\{
    \mu_b^{(\kappa)}\lfloor\sigma_i\rfloor +
       \left[\mu_a\wedge\mu_b(\sigma_i),\mu_b(\sigma_i)\right)
         \right\}
\end{equation*}
Since
\begin{math}
  J_{\mu_{(b,c)},\mu_{(b,d)}}^{(\kappa)}
  = {J'}_{\mu_{(b,c)},\mu_{(b,d)}}^{(\kappa)}
  = \varnothing ,
\end{math}
the map $\Phi_{(b,c),(b,d)}^{(\kappa)}$ is identity on
\begin{equation*}
  J_{\mu_{(a,c)},\mu_{(b,c)}}'^{(\kappa)}\cup
  [\mu_a\vee\mu_b(\kappa-),\mu_{(b,c)}\wedge\mu_{(b,d)}(\kappa)) .
\end{equation*}
The map $\Phi_{(a,c),(b,d)}^{(\kappa)}$ is identity on
\begin{equation*}
  [\mu_a\vee\mu_b(\kappa-),\mu_{(a,c)}\wedge\mu_{(b,d)}(\kappa)),
\end{equation*}
and the same bijection of $\Phi_{(a,c),(b,c)}^{(\kappa)}$ is assigned to
the map $\Phi_{(a,c),(b,d)}^{(\kappa)}$
from $J_{\mu_{(a,c)},\mu_{(b,d)}}^{(\kappa)}$
to $J_{\mu_{(a,c)},\mu_{(b,d)}}'^{(\kappa)}$.
Therefore, \eqref{abcd.comp} must hold.
We can similarly verify the equality with the composition
$\Phi_{(a,d),(b,d)}^{(\kappa)}\circ\Phi_{(a,c),(a,d)}^{(\kappa)}$.
\end{proof}

\subsection{Compositions of RSB's}
\label{rsb.comp.sec}

Let $\mathcal{A}$ be a synchronizable poset,
and let $G_\mathcal{A}$ and $G_{\mathcal{A}^*}$ be the graphs of interlaced
relation on the respective vertex set
$D_\mathcal{A}$ of minimal elements of $\mathcal{A}$
and $D_{\mathcal{A}^*}$ of maximal elements of $\mathcal{A}$.
Then $G_\mathcal{A}$ (and respectively $G_{\mathcal{A}^*}$) has
a locally connected spanning tree $T$ (and respectively $T^*$).
Here we consider the Cartesian product of $T$ and $T'$
restricted on the vertex set
\begin{equation*}
  \{(\alpha,\beta)\in D_\mathcal{A}\times D_{\mathcal{A}^*}:
  \text{$\alpha\le\beta$ in $\mathcal{A}$}\},
\end{equation*}
and denote such graph by $T\,\mbox{\scriptsize$\square$}\,T^*$,
in which a pair $((\alpha,\beta),(\alpha',\beta'))$ is an edge
if $\alpha=\alpha'$ and $(\beta,\beta')\in T^*$;
or if $(\alpha,\alpha')\in T$ and $\beta=\beta'$.

\begin{lemma}\label{g.product}
$T\,\mbox{\scriptsize$\square$}\,T^*$ is connected.
\end{lemma}

\begin{proof}
We prove it by contradiction.
Suppose that $(a_1,b_1),(a_n,b_m)\in T\,\mbox{\scriptsize$\square$}\,T^*$
are disconnected.
Let $(a_1,\ldots,a_n)$ and $(b_1,\ldots,b_m)$
be the path respectively in $T$ and $T^*$.
Then we can find
$(a_i,b_j),(a_{i+1},b_{j+1})\in T\,\mbox{\scriptsize$\square$}\,T^*$
such that
$(a_i,b_{j+1}),(a_{i+1},b_j)\not\in T\,\mbox{\scriptsize$\square$}\,T^*$.
Since $(b_j,b_{j+1})\in T^*$
there exists some $c\in D_{\mathcal{A}}$ satisfying $c<b_j,b_{j+1}$.
The local connectedness of $T$ implies that
$D_{\mathcal{A}}(b_j)$ contains the path $(a_i,\ldots,c)$ in $T$
and $D_{\mathcal{A}}(b_{j+1})$ contains the path $(a_{i+1},\ldots,c)$ in $T$.
Since $a_{i+1}\not\in D_{\mathcal{A}}(b_j)$
and $a_i\not\in D_{\mathcal{A}}(b_{j+1})$,
the path $(a_i\ldots,c',\ldots,a_{i+1})$ contains
some vertex $c'\in D_{\mathcal{A}}(b_j)\cap D_{\mathcal{A}}(b_{j+1})$.
Because $(a_i,a_{i+1})\in T$,
it contradicts the property of $T$ being acyclic.
\end{proof}

In the setting of Remark~\ref{interlace.rem}(a)
we can associate each directed edge
$e=((\alpha,\beta),(\alpha',\beta'))$ in $T\,\mbox{\scriptsize$\square$}\,T^*$
with the interlaced pair $(\mu_{(\alpha,\beta)},\mu_{(\alpha',\beta')})$
and the RSB $\Phi_e$ by Algorithm~\ref{rst.cons}.
We express a path $\Xi=(e_1,\ldots,e_N)$ by a series of directed
edges $e_1,\ldots,e_N$ in $T\,\mbox{\scriptsize$\square$}\,T^*$.
Then we can build the composition $\Phi_{\Xi}$ of RSB's along the path $\Xi$ by
\begin{equation}\label{path.rst}
  \Phi_{\Xi}=\Phi_{e_N}\circ\cdots\circ\Phi_{e_1},
\end{equation}
where we set $\Phi_{\varnothing}(\omega)=\omega$
for the empty path $\Xi=\varnothing$ from a vertex to itself.

In Lemma~\ref{rst.welldef}
consider paths $(a_1,\ldots,a_n)$ and $(b_1,\ldots,b_m)$ respectively
in $T$ and $T^*$.
The proof of Lemma~\ref{g.product} also indicates that
a path $\Xi=(e_1,\ldots,e_N)$ in $T\,\mbox{\scriptsize$\square$}\,T^*$
from $(a_1,b_1)$ to $(a_n,b_m)$ can be constructed with
the minimum size $N=n+m-2$.

\begin{lemma}\label{rst.welldef}
Suppose that $\Xi$ and $\Xi'$ are paths of the size $N=n+m-2$
from $(a_1,b_1)$ to $(a_n,b_m)$ in
$T\,\mbox{\scriptsize$\square$}\,T^*$.
If the choice of bijections is predetermined for Algorithm~\ref{rst.cons}
as described in Lemma~\ref{abcd.rst} then
we have $\Phi_{\Xi}=\Phi_{\Xi'}$.
\end{lemma}

\begin{proof}
We start with the vacuous claim on
$\Xi=\varnothing$ from $(a_1,b_1)$ to itself,
and prove $\Phi_{\Xi}=\Phi_{\Xi'}$ by induction on the size $N=n+m-2$.
If both $\Xi$ and $\Xi'$ end with the same edge
$e=((a_{n},b_{m-1}),(a_{n},b_{m}))$ or
$e'=((a_{n-1},b_{m}),(a_{n},b_{m}))$ then
we have $\Phi_{\Xi}=\Phi_{\Xi'}$ by induction.
Otherwise, we can assume that
$\Xi$ and $\Xi'$ end with the respective edge
$e$ and $e'$.
The path $\Xi'$ ends with the tail
\begin{equation*}
  ((a_{n-k},b_{m-1}),(a_{n-k},b_{m}),\cdots,(a_{n},b_{m}))
\end{equation*}
with some $k\ge 1$.
Since $(a_{n},b_{m-1})\in T\,\mbox{\scriptsize$\square$}\,T^*$,
we can find the complete bipartite Hasse diagram
of the induced subposet of $\mathcal{A}$ on
$\{a_{n-k},\ldots,a_{n},b_{m-1},b_{m}\}$.
By applying Lemma~\ref{abcd.rst} repeatedly we can show
$\Phi_{\Xi'}=\Phi_{\Xi''}$ with the path $\Xi''$ which changes the
tail of $\Xi'$ to
$((a_{n-k},b_{m-1}),\cdots,(a_{n},b_{m-1}),(a_{n},b_{m}))$.
By using the induction hypothesis for the paths $\Xi-e$ and $\Xi''-e$
we have established $\Phi_{\Xi}=\Phi_{\Xi''}$.
\end{proof}

\subsection{Proof of Theorem~\ref{me.sync} for Class W}
\label{rsb.proof.sec}

Suppose that $\mu$ and $\mu'$ are interlaced with $F_\gamma$,
that is, that $\mu$, $\mu'$ and $F_\gamma$ satisfy
$\mu(1)=\mu'(1)=F_\gamma(\tau)$ and
\begin{equation*}
  \mu\vee\mu'(\kappa-)\le F_\gamma(u_*^{(\kappa)})\le F_\gamma(u_1^{(\kappa)})\le\mu\wedge\mu'(\kappa)
\end{equation*}
for each $\kappa\in K$.
In Algorithm~\ref{rstg.cons} we generalize an RSB $\Phi_{\mu,\mu'}^{(\kappa)}$
of the mutually interlaced pair $(\mu,\mu')$,
and construct an RSB $\Phi_{\mu,\mu',\gamma}^{(\kappa)}$ of $(\mu,\mu')$
for $F_\gamma$.

\begin{algorithm}\label{rstg.cons}
We construct a measure-preserving bijection $\Phi_{\mu,\mu',\gamma}^{(\kappa)}$
from $[0,F_\gamma(u_1^{(\kappa)}))$ to itself recursively for $\kappa\in K$.
If $C(\kappa)\neq\varnothing$ then
we have an RSB $\Phi_{\mu,\mu',\gamma}^{(\sigma_i)}$
from $[0,F_\gamma(u_1^{(\sigma_i)}))$ to itself
for each $\sigma_i\in C(\kappa)$ by recursion.
For $\Phi_{\mu,\mu',\gamma}^{(\kappa)}$
we construct \eqref{rst.sigma}
from
$\mu^{(\kappa)}\lfloor\sigma_i\rfloor +
[0,F_\gamma(u_1^{(\sigma_i)}))$
to
$\mu'^{(\kappa)}\lfloor\sigma_i\rfloor +
[0,F_\gamma(u_1^{(\sigma_i)}))$,
and set the identity map
$\Phi_{\mu,\mu',\gamma}^{(\kappa)}$
on the interval
\begin{equation}\label{rstg.id}
  \left[F_\gamma(u_*^{(\kappa)}),F_\gamma(u_1^{(\kappa)})\right).
\end{equation}
Furthermore, we set $I_{\mu,\gamma}^{(\kappa)}$ and $I_{\mu',\gamma}^{(\kappa)}$
by (\ref{map.i})
respectively for $(\mu,F_\gamma)$ and $(\mu',F_\gamma)$,
and extend a measure-preserving bijection
$\Phi_{\mu,\mu',\gamma}^{(\kappa)}$
from $I_{\mu,\gamma}^{(\kappa)}$
to $I_{\mu',\gamma}^{(\kappa)}$ of the equal length
\begin{equation*}
  \lambda(I_{\mu,\gamma}^{(\kappa)}) = \lambda(I_{\mu',\gamma}^{(\kappa)})
  = F_\gamma(u_*^{(\kappa)})
  - \sum_{\sigma_i\in C(\kappa)}F_\gamma(u_1^{(\sigma_i)})
\end{equation*}
\end{algorithm}

\begin{remark}\label{rstg.rem}
(a)
In the base case when $C(\kappa)=\varnothing$,
the map $\Phi_{\mu,\mu',\gamma}^{(\kappa)}$ is identity
on the interval \eqref{rstg.id},
and it is bijective from
$I_{\mu,\gamma}^{(\kappa)}=[0,F_\gamma(u_*^{(\kappa)}))$
to itself.
In Remark~\ref{rst.rem}(a)
we found it different from the base case of Algorithm~\ref{rst.cons}.
(b)
The RSB $\Phi_{(\mu,\mu')}^{(\kappa)}$ of Algorithm~\ref{rst.cons}
is bijective when restricted on $[0,F_\gamma(u_1^{(\kappa)}))$,
and it is viewed as an RSB of $(\mu,\mu')$ for $F_\gamma$.
(c)
Suppose that $\mu$, $\mu'$ and $\mu''$ are interlaced with $F_\gamma$.
Then the composition $\Phi_{\mu',\mu'',\gamma}^{(\kappa)}\circ\Phi_{\mu,\mu',\gamma}^{(\kappa)}$
is an RSB of $(\mu,\mu'')$ for $F_\gamma$ by Algorithm~\ref{rstg.cons}.
However, it may not be
an RSB of $(\mu,\mu'')$ by Algorithm~\ref{rst.cons}.
\end{remark}

\begin{lemma}\label{map.unique}
Let $X_{\mu,\gamma}^{(\kappa)}$ and $X_{\mu',\gamma}^{(\kappa)}$
be the recursive inverse transforms from $[0,F_{\gamma}(u_1^{(\kappa)}))$ to $S^{(\kappa)}$
by Algorithm~\ref{map.alg}.
Then we have
\begin{math}
  X_{\mu,\gamma}^{(\kappa)}
  =X_{\mu',\gamma}^{(\kappa)}\circ\Phi_{\mu,\mu',\gamma}^{(\kappa)} .
\end{math}
\end{lemma}

\begin{proof}
For the construction of $X_{\mu',\gamma}^{(\kappa)}$
we observe
(\ref{map.recursion}) in Algorithm~\ref{map.alg}
and (\ref{map.i}) in Remark~\ref{map.rem},
and obtain the recursive representation
of $X_{\mu',\gamma}^{(\kappa)}\circ\Phi^{(\kappa)}$ by
\begin{equation*}
  \begin{cases}
    X_{\mu',\gamma}^{(\sigma_i)}
    (\Phi^{(\sigma_i)}(\omega-\mu_{\beta}^{(\kappa)}\lfloor\sigma_i\rfloor))
    & \mbox{if $\omega\in
      \mu^{(\kappa)}\lfloor\sigma_i\rfloor
      + \left[0,F_\gamma(u_1^{(\sigma_i)})\right)$} \\
      & \mbox{\hspace{5ex} for some $\sigma_i\in C(\kappa)$;} \\
      u_*^{(\kappa)}
      & \mbox{if $\omega\in I_{\mu,\gamma}^{(\kappa)}$;} \\
      X_{\mu',\gamma}^{(\kappa)}(\omega)
      & \mbox{if $\omega\in
        \left[F_\gamma(u_*^{(\kappa)}),F_\gamma(u_1^{(\kappa)})\right)$}
  \end{cases}
\end{equation*}
which equals the inverse transform $X_{\mu,\gamma}^{(\kappa)}$ by induction.
\end{proof}

Suppose that $\mathcal{A}$ is a synchronizable poset,
and that $(F_\alpha:\alpha\in\mathcal{A})$ is a stochastically monotone system
of probability distribution functions.
Then we can consider the graph product $T\,\mbox{\scriptsize$\square$}\,T^*$
of Lemma~\ref{g.product},
and associate an edge
$e=((\alpha,\beta),(\alpha',\beta'))\in T\,\mbox{\scriptsize$\square$}\,T^*$
with the interlaced pair $(\mu_{(\alpha,\beta)},\mu_{(\alpha',\beta')})$
of Remark~\ref{interlace.rem}(a).
Then we construct the RSB $\Phi_{\mu_{(\alpha,\beta)},\mu_{(\alpha',\beta')}}$
by Algorithm~\ref{rst.cons},
and simply write $\Phi_e$.

In Lemma~\ref{map.welldef}
we fix $(a_1,b_1)\in T\,\mbox{\scriptsize$\square$}\,T^*$.
By Lemma~\ref{rst.welldef} we can set the well-defined composition \eqref{path.rst}
along a path $\Xi=(e_1,\ldots,e_N)$ of minimum size from $(a_1,b_1)$ to
$(\alpha,\beta)\in T\,\mbox{\scriptsize$\square$}\,T^*$,
and simply write $\Phi_{(\alpha,\beta)}$ for $\Phi_{\Xi}$ of \eqref{path.rst}.
For each $(\alpha,\beta)\in T\,\mbox{\scriptsize$\square$}\,T^*$
and each $\gamma\in\mathcal{A}$ satisfying $\alpha\le\gamma\le\beta$,
we write $X_{(\alpha,\beta),\gamma}$ for
the inverse transform $X_{\mu_{(\alpha,\beta)},\gamma}^{(1)}$ of Lemma~\ref{map.unique}.

\begin{lemma}\label{map.welldef}
Let $(\alpha,\beta),(\alpha',\beta')\in T\,\mbox{\scriptsize$\square$}\,T^*$
and $\gamma\in\mathcal{A}$.
If $\alpha,\alpha'\le\gamma\le\beta,\beta'$
then
\begin{equation}\label{x.gamma}
  X_{(\alpha,\beta),\gamma}\circ\Phi_{(\alpha,\beta)}
  =X_{(\alpha',\beta'),\gamma}\circ\Phi_{(\alpha',\beta')}
\end{equation}
\end{lemma}

\begin{proof}
If $\Xi$ is a path of minimum size from $(\alpha,\beta)$
to $(\alpha',\beta')$ then each $(\alpha'',\beta'')\in\Xi$ satisfies
$\alpha''\le\gamma\le\beta''$.
Hence, it suffices to show \eqref{x.gamma} for
$((\alpha,\beta),(\alpha',\beta'))\in T\,\mbox{\scriptsize$\square$}\,T^*$.
Consider the case when $\beta=\beta'$ and $(\alpha,\alpha')\in T$.
Without loss of generality
we have paths $\xi=(a_1,\ldots,a_{n-1},\alpha)$
and $\xi'=\xi+(\alpha,\alpha')$ in $T$,
and therefore, we can find a path $\Xi$ of minimum size
from $(a_1,b_1)$ to $(\alpha,\beta)$,
and $\Xi' = \Xi+((\alpha,\beta),(\alpha',\beta'))$
from $(a_1,b_1)$ to $(\alpha',\beta')$.
By Lemma~\ref{map.unique} we have
\begin{equation*}
  X_{(\alpha,\beta),\gamma}\circ\Phi_{\Xi}
  = X_{(\alpha',\beta'),\gamma}\circ\Phi_{\mu_{(\alpha,\beta)},\mu_{(\alpha',\beta')}}\circ\Phi_{\Xi}
  = X_{(\alpha',\beta'),\gamma}\circ\Phi_{\Xi'}
\end{equation*}
which implies \eqref{x.gamma} for $\beta=\beta'$.
The case when $\alpha=\alpha'$ and $(\beta,\beta')\in T^*$
is similarly handled.
\end{proof}

By Lemma~\ref{map.welldef} we can construct a well-defined inverse transform
$X_{(\alpha,\beta),\gamma}\circ\Phi_{(\alpha,\beta)}$ for
each $\gamma\in\mathcal{A}$
by arbitrarily choosing $(\alpha,\beta)\in T\,\mbox{\scriptsize$\square$}\,T^*$
which satisfies $\alpha\le\gamma\le\beta$,
and denote it by $X_\gamma$.
Since $\Phi_{(\alpha,\beta)}$ is a measure-preserving bijection from $[0,1)$
to itself, by Proposition~\ref{map.prop} the inverse transform $X_\gamma$ realizes
$F_\gamma$ in the sense of \eqref{map.realize}.
Suppose that $\gamma\le\gamma'$.
Then we can find some $(\alpha,\beta)\in T\,\mbox{\scriptsize$\square$}\,T^*$
satisfying $\alpha\le\gamma\le\gamma'\le\beta$,
and by applying Proposition~\ref{map.prop} for $F_{\gamma}\preceq F_{\gamma'}$
we obtain
\begin{equation*}
  X_{\gamma}(\omega)=X_{(\alpha,\beta),\gamma}\circ\Phi_{(\alpha,\beta)}(\omega)
  \le X_{(\alpha,\beta),\gamma'}\circ\Phi_{(\alpha,\beta)}(\omega)
  =X_{\gamma'}(\omega),
  \quad\omega\in [0,1)
\end{equation*}
Thus,
we found $(F_\alpha:\alpha\in\mathcal{A})$ realizably monotone,
which completes the proof of Theorem~\ref{me.sync} in the case of (i).

\subsection{Proof of Theorem~\ref{me.sync} for Class W$_\star$}
\label{wstar.sec}

Consider the case (ii) of Theorem~\ref{me.sync} where
$\mathcal{S}$ is in Class W$_\star$ and $\mathcal{A}$ is a synchronizable poset for
Class W$_\star$.
We continue to assume the stochastically monotone system $(F_\alpha:\alpha\in\mathcal{A})$
of probability distribution functions.
Then we can associate a directed edge $e=(\alpha,\alpha')\in T$
with the interlaced pair $(\mu_{\alpha},\mu_{\alpha'})$ by \eqref{a.dist}
as in Remark~\ref{interlace.rem}(b),
and construct the RSB $\Phi_e$ by Algorithm~\ref{rst.cons}.
Let $T$ be a locally connected spanning tree on the graph
$G_{\mathcal{A}}$ of interlaced relation,
and $a_1\in T$ be a fixed vertex.
Thus, we obtain a unique path $\xi=(e_1,\ldots,e_N)$
from $a_1$ to a vertex $\alpha\in T$
with directed edges $e_1,\ldots,e_N\in T$,
and define the composition $\Phi_{\xi}$ by
\begin{equation*}
  \Phi_{\xi}=\Phi_{e_N}\circ\cdots\circ\Phi_{e_1},
\end{equation*}
where we stipulate $\Phi_{\varnothing}(\omega)=\omega$ when $\xi=\varnothing$.
Since $\Phi_{\xi}$ is uniquely determined for each $\alpha\in T$,
we simply write $\Phi_\alpha$ for $\Phi_{\xi}$.
In Lemma~\ref{map2.welldef}
for $\alpha\in T$ and $\gamma\in\mathcal{A}$ satisfying $\alpha\le\gamma$,
we write $X_{\alpha,\gamma}$ for the inverse transform $X_{\mu_\alpha,\gamma}^{(1)}$ of
Lemma~\ref{map.unique}.

\begin{lemma}\label{map2.welldef}
Let $\alpha,\alpha'\in T$ and $\gamma\in\mathcal{A}$.
If $\alpha,\alpha'\le\gamma$ then
\begin{equation}\label{x2.gamma}
  X_{\alpha,\gamma}\circ\Phi_{\alpha}=X_{\alpha',\gamma}\circ\Phi_{\alpha'}.
\end{equation}
\end{lemma}

\begin{proof}
Since a unique path in $T$ from $\alpha$ to $\alpha'$ has an upper bound $\gamma$ in
$\mathcal{A}$,
it suffices to show \eqref{x2.gamma} for $(\alpha,\alpha')\in T$.
As in the proof of Lemma~\ref{map.welldef}
we have a path $\xi$ from $a_1$ to $\alpha$,
and a path $\xi'=\xi+(\alpha,\alpha')$ from $a_1$ to $\alpha'$.
By Lemma~\ref{map.unique} we obtain
\begin{equation*}
  X_{\alpha,\gamma}\circ\Phi_{\xi}
  =X_{\alpha',\gamma}\circ\Phi_{\mu_{\alpha},\mu_{\alpha'}}\circ\Phi_{\xi}
  =X_{\alpha',\gamma}\circ\Phi_{\xi'}
\end{equation*}
as desired.
\end{proof}

For each $\gamma$ the inverse transform $X_{\alpha,\gamma}\circ\Phi_{\alpha}$,
denoted by $X_\gamma$,
is well-defined by Lemma~\ref{map2.welldef}
regardless of $\alpha\in D_{\mathcal{A}}(\gamma)$,
and it realizes $F_\gamma$.
For $\gamma\le\gamma'$
we can find some $\alpha\in T$ such that $\alpha\le\gamma$,
and obtain by Proposition~\ref{map.prop}
\begin{equation*}
  X_{\gamma}(\omega)=X_{\alpha,\gamma}\circ\Phi_{\alpha}(\omega)
  \le X_{\alpha,\gamma'}\circ\Phi_{\alpha}(\omega)=X_{\gamma'}(\omega),
  \quad\omega\in [0,1),
\end{equation*}
which implies that
$(F_\alpha:\alpha\in\mathcal{A})$ is realizably monotone.
Thus, Theorem~\ref{me.sync} holds in the case of (ii),
and the case of (iii) can be dually established.

\section{Monotonicity inequivalence for Class W$_\star$}\label{mi.sec}
\setcounter{equation}{0}\setcounter{figure}{0}

In~\cite{fm2002}
we conjectured that the synchronizability of $\mathcal{A}$
is somewhat necessary for the monotonicity equivalence.
In this section we will complete the proof of

\begin{proposition}\label{main.claim}
If $\mathcal{A}$ is not synchronizable for Class W$_{\star}$
then there is some $n\ge 1$ so that
monotonicity equivalence fails for $(\mathcal{A},\mathcal{S})$
with $(n+2)$-legged W$_{\star}$-poset $\mathcal{S}$
of Figure~\ref{wposet}.
\end{proposition}

\begin{figure}[h]\begin{center}
\tikzset{circle dotted/.style={line cap=round, line width=1pt, dash pattern=on 0pt off 10pt}}
\begin{tikzpicture}
  \node [mynode,label=above:$z$] (z) at (0,0) {};
  \node [mynode,label=below:$y_0$] (y0) at (-3,-1) {};
  \node [mynode,label=below:$y_1$] (y1) at (-2,-1) {};
  \node [mynode,label=below:$y_n$] (yn) at ( 2,-1) {};
  \node [mynode,label=below:$y_*$] (y*) at ( 3,-1) {};
  \draw[thick]
  (z) -- (y0)
  (z) -- (y1)
  (z) -- (yn)
  (z) -- (y*);
  \draw [circle dotted] (1,-1) -- (-1,-1);
\end{tikzpicture}
\caption{$(n+2)$-legged W$_{\star}$-poset}
\label{wposet}
\end{center}\end{figure}
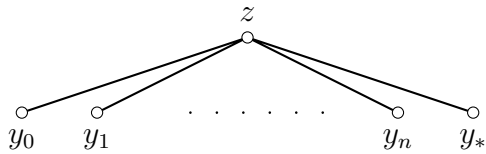

\subsection{Spanning trees and Kruskal algorithm}

Consider a connected graph $G = (D,E)$ with vertex set $D$ and edge set $E$.
Having fixed the vertex set $D$,
we can construct a new supergraph $G+e$ by adding a new edge $e$ to $E$,
or a subgraph $G-e$ by deleting an edge $e$ from $E$.
We can introduce a nonnegative function $\theta$ on the edge set $E$,
and call $\theta(e)$ the \emph{length} of an edge $e$.
For any subgraph $H$ of $G$ we can define the \emph{weight} of $H$ by
\begin{equation*}
  \theta(H) = \sum_{e\in H}\theta(e)
\end{equation*}
Kruskal~\cite{kruskal1956} showed that Algorithm~\ref{kruskal} constructs a minimum weight
spanning tree $T$ of $G$.

\begin{algorithm}[Kruskal algorithm]\label{kruskal}
Let $\mathcal{E} = (e_1,\ldots,e_M)$ be a linearly ordered edge set
by the length so that  
$\theta(e_i)\le\theta(e_j)$ whenever $e_i < e_j$ in $\mathcal{E}$.
It starts from $T_0 = (D,\varnothing)$,
and constructs $T_i$ recursively from $T_{i-1}$ as follows:
If $T_{i-1}+e_i$ forms a loop then we set $T_i = T_{i-1}$;
otherwise, we update $T_i = T_{i-1}+e_i$.
The update is repeated recursively until $T_i$ becomes a spanning tree.
\end{algorithm}

As the construction of Algorithm~\ref{kruskal} suggests,
a minimum weight spanning tree may not be uniquely determined unless the lengths are all distinct.
In Proposition~\ref{descend.alg}
if $T$ is a spanning tree then we can construct a new spanning tree $T' = T+e-f$ by adding a new
edge $e$ and deleting an edge $f$ belonging to the loop formed by $T+e$.
We can show that a series of such constructions can always descend to a
particular minimum weight spanning tree.

\begin{proposition}\label{descend.alg}
Let $T_k$ be a minimum weight spanning tree of $G$.
Starting from any spanning tree $T_0$,
we can find a series of constructions, $T_i = T_{i-1}+e_i-f_i$,
and terminate at the spanning tree $T_k$
while it maintains $\theta(T_i)\le\theta(T_{i-1})$ for $i=1,\ldots,k$.
\end{proposition}

\begin{proof}
By the lengths $\theta(e_i)$'s we can linearly order the edge set
$\mathcal{E}=(e_1,\ldots,e_M)$,
in which a pair $\{e_i,e_j\}$ of the same length
satisfies $e_i < e_j$ if $e_i\in T_k$ or $e_j \in T_0\setminus T_k$.
The subgraph $T_i\cap\{e_1,\ldots,e_i\}$
restricted on the edge set $\{e_1,\ldots,e_i\}$
will be identical to the subgraph $T_k\cap\{e_1,\ldots,e_i\}$
of Algorithm~\ref{kruskal}
in the following construction of $T_i$:
If $e_i\in T_{i-1}$ then we keep $T_i = T_{i-1}$ unchanged.
Otherwise, $T_{i-1}+e_i$ forms a unique loop $C$.
If $C\subseteq\{e_1,\ldots,e_i\}$
then $T_k$ contains $C-e_i$ but it cannot contain $e_i$;
thus, we keep $T_i = T_{i-1}$.
If $C$ contains an edge $e_j > e_i$ in $\mathcal{E}$
then $T_k$ contains $e_i$ but it does not contain $e_j$;
thus, we update $T_i = T_{i-1}+e_i-e_j$
and obtain $\theta(T_i)\le\theta(T_{i-1})$.
Consequently a subsequence of the updates of $T_i$ becomes a desired series of constructions.
\end{proof}

Let $\mathcal{A}$ be a poset,
and consider the graph
$G_{\mathcal{A}}=(D_{\mathcal{A}},E_{\mathcal{A}})$
of interlaced relation
on the set $D_{\mathcal{A}}$ of minimal elements in $\mathcal{A}$.
We define the principal up-set $\mathcal{A}(\beta)$ of $\beta\in\mathcal{A}$ by
\begin{equation*}
  \mathcal{A}(\beta) = \{\alpha\in\mathcal{A}: \mbox{$\beta\le\alpha$ in $\mathcal{A}$}\}
\end{equation*}
and for each $(\beta,\beta')\in E_{\mathcal{A}}$
\begin{equation}\label{theta.set}
  \Theta(\beta,\beta') = \{\gamma\in D_{\mathcal{A}}:
  \mathcal{A}(\beta)\cap\mathcal{A}(\beta')\subseteq\mathcal{A}(\gamma)\}
\end{equation}
Then we can assign a length $\theta(\beta,\beta')$ of the edge
$(\beta,\beta')$ by
\begin{equation}\label{theta}
  \theta(\beta,\beta') = |\Theta(\beta,\beta')|
\end{equation}
where $|\cdot|$ denotes the cardinality.
Using $D_{\mathcal{A}}(\alpha)$ of \eqref{d.set}
we can express the vertex set $\Theta(\beta,\beta')$ equivalently by
\begin{equation}\label{theta.cap}
  \Theta(\beta,\beta')
  = \bigcap_{\alpha\in\mathcal{A}(\beta)\cap\mathcal{A}(\beta')}
  D_{\mathcal{A}}(\alpha) .
\end{equation}

Suppose that $T$ is a spanning tree of $G_{\mathcal{A}}$.
Then the graph $T+(b_0,b_n)$ with an edge $(b_0,b_n)\not\in T$
forms a loop $C = (b_0,b_1,\ldots,b_n,b_0)$.
In Lemma~\ref{l.connected}
we arbitrarily choose an edge $(b_{i-1},b_i)$ from the path $C-(b_0,b_n)$,
and construct a spanning tree $T' = T+(b_0,b_n)-(b_{i-1},b_i)$.

\begin{lemma}\label{l.connected}
If $\theta(T')\le\theta(T)$ and $T$ is locally connected
then $\theta(T')=\theta(T)$ and $T'$ is also locally connected.
\end{lemma}

\begin{proof}
Since $T$ is locally connected,
the intersection formulation of \eqref{theta.cap}
implies that
$\Theta(b_0,b_n)$ contains the path $C-(b_0,b_n)$.
In particular we have $b_{i-1},b_i\in\Theta(b_0,b_n)$,
and by \eqref{theta.set} we find
\begin{equation*}
  \mathcal{A}(b_0)\cap\mathcal{A}(b_n)\subseteq\mathcal{A}(b_{i-1})\cap\mathcal{A}(b_i).
\end{equation*}
Since $\theta(b_0,b_n)\le\theta(b_{i-1},b_i)$, by \eqref{theta.set}
we obtain
$\Theta(b_0,b_n)=\Theta(b_{i-1},b_i)$,
and therefore, $\theta(b_0,b_n)=\theta(b_{i-1},b_i)$.
In the rest of proof
we claim that the subtree $T'\cap D_{\mathcal{A}}(\alpha)$ of $T'$
restricted on the vertex set $D_{\mathcal{A}}(\alpha)$
is connected for every $\alpha\in\mathcal{A}$.
Suppose that $b_{i-1}\not\in D_{\mathcal{A}}(\alpha)$
or $b_{i}\not\in D_{\mathcal{A}}(\alpha)$.
Then $D_{\mathcal{A}}(\alpha)$ contains only one of the two components
of $T-(b_{i-1},b_i)$,
and therefore,
$T'\cap D_{\mathcal{A}}(\alpha)=T\cap D_{\mathcal{A}}(\alpha)$
is connected.
Suppose that $b_{i-1},b_i\in D_{\mathcal{A}}(\alpha)$.
Then we have
\begin{equation*}
  D_{\mathcal{A}}(\alpha)\supseteq\Theta(b_{i-1},b_i)
  =\Theta(b_0,b_n)\supseteq C-(b_0,b_n)
\end{equation*}
and therefore,
\begin{math}
  T'\cap D_{\mathcal{A}}(\alpha)
  =T\cap D_{\mathcal{A}}(\alpha)+(b_0,b_n)-(b_{i-1},b_i)
\end{math}
is connected.
\end{proof}

We can start with a locally connected spanning tree $T_0$
in Proposition~\ref{descend.alg},
and find the following corollary to
Lemma~\ref{l.connected}.

\begin{corollary}\label{l.conn.cor}
Let $\mathcal{A}$ be a synchronizable poset,
and let $T$ be a spanning tree of $G_{\mathcal{A}}$.
Then $T$ is locally connected if and only if
$T$ attains the minimum weight by the length (\ref{theta}).
\end{corollary}

\subsection{Proof of Proposition~\ref{main.claim}}

Let $T_0$ be a forest subgraph (i.e., an acyclic subgraph)
of the graph $G_{\mathcal{A}} = (D_{\mathcal{A}}, E_{\mathcal{A}})$,
and let $b_0\in T_0$.
An induced $(2l+3)$-fence subposet (Figure~\ref{fence.fig}; cf.~\cite{duffus1992,stanleyvol1}) 
\begin{equation}\label{t.fence}
  b_0<a_1>b_{1}<\ldots>b_{l}<a_{l+1}>b_*
\end{equation}
of $\mathcal{A}$ is said to be supported by $T_0$ from $b_0\in T_0$ to
$b_*\not\in T_0$
if $l\ge 1$ and there exists
a path $(x_0,x_1,\ldots,x_{j_l})$ in $T_0$ from $x_0=b_0$ to $x_{j_l}=b_l$
in which $b_0,b_{1},\ldots,b_{l}$
is identified as a subsequence $x_0,x_{j_1},\ldots,x_{j_l}$
with indices $0 < j_1 < \cdots < j_l$.
In the fence poset all the comparable pairs are expressed in \eqref{t.fence}
and no other pairs are comparable.

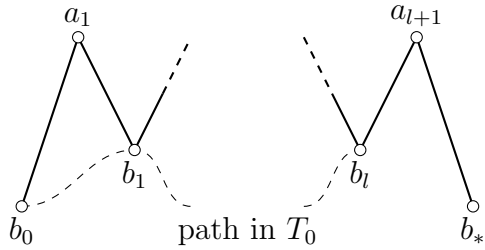
\begin{figure}[h]\begin{center}
\begin{tikzpicture}[xscale=1.5,yscale=1.5]
\draw[dashed] (0,0) to [out=0,in=180] (1,0.5) to [out=0,in=180] (1.5,0);
\draw[dashed] (2.5,0) to [out=0,in=180] (3,0.5);
\draw[thick] (0,0) -- (0.5,1.5) -- (1,0.5) -- (1.25,1);
\draw[dashed, thick] (1.25,1) -- (1.5,1.5);
\draw[dashed, thick] (2.5,1.5) -- (2.75,1);
\draw[thick] (2.75,1) -- (3,0.5) -- (3.5,1.5) -- (4,0);

\draw[fill=white] (0,0) circle [radius=0.05];
\node[below] at (0,0) {$b_0$};
\draw[fill=white] (1,0.5) circle [radius=0.05];
\node[below] at (1,0.5) {$b_1$};
\draw[fill=white] (3,0.5) circle [radius=0.05];
\node[below] at (3,0.5) {$b_l$};
\draw[fill=white] (4,0) circle [radius=0.05];
\node[below] at (4,0) {$b_*$};
\draw[fill=white] (0.5,1.5) circle [radius=0.05];
\node[above] at (0.5,1.5) {$a_1$};
\draw[fill=white] (3.5,1.5) circle [radius=0.05];
\node[above] at (3.5,1.5) {$a_{l+1}$};
\node[below] at (2,0) {path in $T_0$};
\end{tikzpicture}

\caption{
  $(2l+3)$-fence supported by $T_0$ from $b_0\in T_0$ to $b_*\not\in T_0$.
}\label{fence.fig}
\end{center}\end{figure}

In this subsection for the proof of Proposition~\ref{main.claim}
we assume that $\mathcal{A}$ is not synchronizable.
Let $T$ be a minimum weight spanning tree of
the graph $G_{\mathcal{A}} = (D_{\mathcal{A}}, E_{\mathcal{A}})$ by the length~(\ref{theta}),
and choose $a_0\in\mathcal{A}$ so that the subgraph
$T\cap D_{\mathcal{A}}(a_0)$ is disconnected.
For Lemma~\ref{fence.lem}
we can find a vertex $x_0=b_0$ in one component of $T\cap D_{\mathcal{A}}(a_0)$
such that a path $(x_0,\ldots,x_m)$ in $T$ traverses on edges not in $T\cap D_{\mathcal{A}}(a_0)$
and comes back to a vertex $x_m$ in another component of $T\cap D_{\mathcal{A}}(a_0)$.
Since $x_0$ and $x_m$ are disconnected in $T\cap D_{\mathcal{A}}(a_0)$,
we also find $m\ge 2$.
We define a forest by
\begin{equation*}
  T_0 = T\cap((D_{\mathcal{A}}\setminus D_{\mathcal{A}}(a_0))\cup\{b_0\})
\end{equation*}
which contains all the edges of $T$ not ending on
$D_{\mathcal{A}}(a_0)\setminus\{b_0\}$.

\begin{lemma}\label{fence.lem}
There exists an induced $(2l+3)$-fence subposet \eqref{t.fence} of $\mathcal{A}$
supported by $T_0$ from $b_0=x_0$ to $b_*=x_m$.
\end{lemma}

\begin{proof}
First we claim by contradiction that
\begin{equation}\label{pick.y}
  \mathcal{A}(x_{i-1})\cap\mathcal{A}(x_i)
  \setminus
  (\mathcal{A}(b_0)\cap\mathcal{A}(b_*))
  \neq\varnothing
\end{equation}
for each $i=1,\ldots,m$.
Suppose that \eqref{pick.y} is the empty set for some $i$.
Since $m\ge 2$, we have
$a_0\not\in\mathcal{A}(x_{i-1})$ or $a_0\not\in\mathcal{A}(x_i)$,
and therefore,
$x_{i-1}\not\in\Theta(b_0,b_*)$ or $x_i\not\in\Theta(b_0,b_*)$.
Thus, $\Theta(b_0,b_*)$ must be strictly smaller than $\Theta(x_{i-1},x_i)$,
and consequently,
the spanning tree $T+(b_0,b_*)-(x_{i-1},x_i)$ has a strictly smaller weight,
contradicting the minimum weight of $T$.
Thus, we can choose an element $y_i$ of \eqref{pick.y} for each $i=1,\ldots,m$

Starting from $b_0=x_0$ we can determine
$a_1=y_{i_1}$ of the largest index $i_1$ comparable with $b_0$.
Provided $a_1$,
we determine $b_1=x_{j_1}$ of the largest index $j_1$
comparable with $a_1$ and so forth
until the last element $a_{l+1} = y_{i_{l+1}}$ is comparable with $b_*=x_m$.
The choice of $y_i$'s guarantees $l\ge 1$,
and the completed sequence of \eqref{t.fence}
forms the induced fence of $\mathcal{A}$ as desired.
\end{proof}

In the setting of Lemma~\ref{fence.lem}
we form the collection $\Gamma_0$ of induced fence subposets of
$\mathcal{A}$ supported by $T_0$ from $b_0$ to some element
$b_*\in D_{\mathcal{A}}(a_0)$.
Each of the maximal elements $a_1,\ldots,a_{l+1}$ of a fence $\mathcal{F}\in\Gamma_0$
cannot be comparable with both $b_0$ and $b_*$ in $\mathcal{A}$,
and therefore, it is incomparable with $a_0$
since $b_1,\ldots,b_l\not\in D_{\mathcal{A}}(a_0)$.
Thus, the subposet of $\mathcal{A}$ induced on $\mathcal{F}\cup\{a_0\}$
becomes an $(2l+4)$-crown poset (cf~\cite{duffus1992}).

In order to construct a counter example for Proposition~\ref{main.claim}
we enumerate $\Gamma_0$ and express each fence
$\mathcal{F}^{(i)}\in\Gamma_0$
by
\begin{equation*}
  b_0<a_1^{(i)}>b_1^{(i)}<\ldots>b_{l_i}^{(i)}<a_{l_i+1}^{(i)}>b_*^{(i)} .
\end{equation*}
For each $\mathcal{F}^{(j)}\in\Gamma_0$
we count $\mathcal{F}^{(i)}$'s
satisfying $b_*^{(j)}< a_1^{(i)}$ in $\mathcal{A}$,
and place $\mathcal{F}^{(1)}$
so that $M = |\{i: b_*^{(1)}< a_1^{(i)}\}|$ attains the maximum of such numbers.
Furthermore, we place
$\mathcal{F}^{(2)},\ldots,\mathcal{F}^{(n)}$
in such a way that
\begin{equation*}
  C_1=\{a_1^{(1)},a_1^{(2)},\ldots,a_1^{(n)}\}
\end{equation*}
is identified with the collection
of minimal elements in
the induced subposet on
\begin{equation}\label{a1.set}
  \{a_1^{(1)}\}\cup\{a_1^{(i)}: b_*^{(1)}< a_1^{(i)}\}.
\end{equation}
Note that $a_1^{(1)}$ is always minimal in such subposet
because $a_1^{(1)}$ is not comparable with $b_*^{(1)}$.
Then we define
\begin{equation*}
  B_* = \{b_*^{(1)},\ldots,b_*^{(n)}\}
\end{equation*}
whose cardinality may be less than $n$.
We further introduce
\begin{align*}
  B_1 &= \bigcup_{i=1}^n\{b_1^{(i)},\ldots,b_{l_i}^{(i)}\};
  \\
  C_2 &= \bigcup_{i=1}^n\{a_2^{(i)},\ldots,a_{l_i+1}^{(i)}\}.
\end{align*}
We have $B_1\subseteq D_{\mathcal{A}}\setminus D_{\mathcal{A}}(a_0)$
and $B_*\subseteq D_{\mathcal{A}}(a_0)$,
and therefore, $\{b_0\}$, $B_1$, and $B_*$ are mutually disjoint.
Similarly $\{a_0\}$, $C_1$, and $C_2$ are disjoint since $b_0$ is incomparable
with each $\alpha\in C_2$;
in particular, if $a_1^{(i)}\in C_1$ and $\alpha\in C_2$ are comparable
then we must have $\alpha< a_1^{(i)}$.

\begin{figure}[h]\begin{center}
\begin{tikzpicture}[xscale=1.5,yscale=1.5]
\draw[thick]
  (0,0) -- (1,1)
  (0,0) -- (2,1)
  (0,0) -- (3,1)
  (1,0) -- (0,1)
  (1,0) -- (2,1)
  (1,0) -- (3,1)
  (2,0) -- (0,1)
  (2,0) -- (1,1)
  (2,0) -- (3,1)
  (3,0) -- (0,1)
  (3,0) -- (1,1)
  (3,0) -- (2,1);

\draw[fill=white] (0,0) circle [radius=0.05];
\node[below] at (0,0) {$b_1$};
\draw[fill=white] (1,0) circle [radius=0.05];
\node[below] at (1,0) {$b_*^{(1)}$};
\draw[fill=white] (2,0) circle [radius=0.05];
\node[below] at (2,0) {$b_*^{(2)}$};
\draw[fill=white] (3,0) circle [radius=0.05];
\node[below] at (3,0) {$b_0$};
\draw[fill=white] (0,1) circle [radius=0.05];
\node[above] at (0,1) {$a_0$};
\draw[fill=white] (1,1) circle [radius=0.05];
\node[above] at (1,1) {$a_1^{(1)}$};
\draw[fill=white] (2,1) circle [radius=0.05];
\node[above] at (2,1) {$a_1^{(2)}$};
\draw[fill=white] (3,1) circle [radius=0.05];
\node[above] at (3,1) {$a_2$};
\end{tikzpicture}
\caption{
  Standard example of $4$-dimensional poset (cf. \cite{trotter})
}\label{4d.poset}
\end{center}\end{figure}
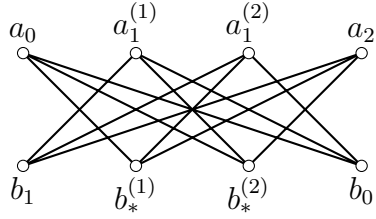

\begin{example}\label{4d.example}
Let $\mathcal{A}$ be a poset of Figure~\ref{4d.poset}.
Since $\Theta(\beta,\beta')=\{\beta,\beta'\}$ for each
$(\beta,\beta')\in E_{\mathcal{A}}$,
a spanning tree $T$ with edges
$(b_0,b_1)$, $(b_1,b_*^{(1)})$ and $(b_1,b_*^{(2)})$
attains the minimum weight of $6$.
$T\cap D_{\mathcal{A}}(a_0)$ is not connected,
and therefore, $\mathcal{A}$ is not synchronizable
by Corollary~\ref{l.conn.cor}.
We can choose $b_0\in D_{\mathcal{A}}(a_0)$
for which $T_0$ consists of one edge $(b_0,b_1)$.
Then we can enumerate two fences supported by $T_0$
from $b_0$ to $b_*^{(i)}\in D_{\mathcal{A}}(a_0)$ as follows:
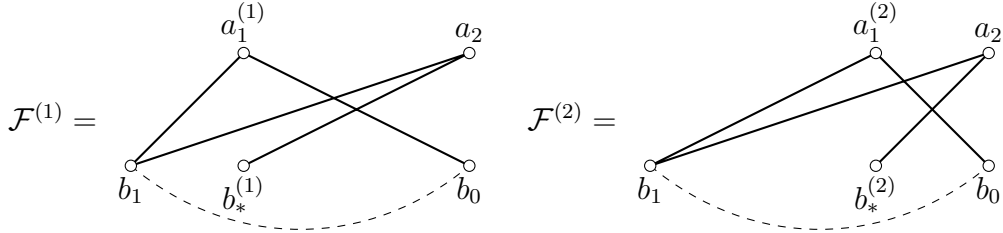
\begin{center}
$\mathcal{F}^{(1)} =$
\begin{minipage}{30ex}
\begin{tikzpicture}[xscale=1.5,yscale=1.5]
\draw[thick]
  (0,0) -- (1,1)
  (0,0) -- (3,1)
  (1,0) -- (3,1)
  (3,0) -- (1,1);
\draw[dashed] (0,0) to [out=-40,in=220] (3,0);

\draw[fill=white] (0,0) circle [radius=0.05];
\node[below] at (0,0) {$b_1$};
\draw[fill=white] (1,0) circle [radius=0.05];
\node[below] at (1,0) {$b_*^{(1)}$};
\draw[fill=white] (3,0) circle [radius=0.05];
\node[below] at (3,0) {$b_0$};
\draw[fill=white] (1,1) circle [radius=0.05];
\node[above] at (1,1) {$a_1^{(1)}$};
\draw[fill=white] (3,1) circle [radius=0.05];
\node[above] at (3,1) {$a_2$};
\end{tikzpicture}
\end{minipage}
$\mathcal{F}^{(2)} =$
\begin{minipage}{30ex}
\begin{tikzpicture}[xscale=1.5,yscale=1.5]
\draw[thick]
  (0,0) -- (2,1)
  (0,0) -- (3,1)
  (2,0) -- (3,1)
  (3,0) -- (2,1);
\draw[dashed] (0,0) to [out=-40,in=220] (3,0);

\draw[fill=white] (0,0) circle [radius=0.05];
\node[below] at (0,0) {$b_1$};
\draw[fill=white] (2,0) circle [radius=0.05];
\node[below] at (2,0) {$b_*^{(2)}$};
\draw[fill=white] (3,0) circle [radius=0.05];
\node[below] at (3,0) {$b_0$};
\draw[fill=white] (2,1) circle [radius=0.05];
\node[above] at (2,1) {$a_1^{(2)}$};
\draw[fill=white] (3,1) circle [radius=0.05];
\node[above] at (3,1) {$a_2$};
\end{tikzpicture}
\end{minipage}
\end{center}
Here we find
$\{a_1^{(i)}:b_*^{(1)}<a_1^{(i)}\}=\{a_1^{(2)}\}$ maximal
and $C_1=\{a_1^{(1)},a_1^{(2)}\}$ identified
with the collection of minimal elements in \eqref{a1.set}.
We obtain
$B_*=\{b_*^{(1)},b_*^{(2)}\}$,
$B_1=\{b_1\}$ and $C_2=\{a_2\}$.
\end{example}

\begin{lemma}\label{b.star}
For each $b_*\in D_{\mathcal{A}}(a_0)\setminus\{b_0\}$
we have $|\mathcal{A}(b_*)\cap C_1|\le n-1$
or $\mathcal{A}(b_*)\cap C_2=\varnothing$.
\end{lemma}

\begin{proof}
We prove it by contradiction.
Suppose that there exists some $b_*\in D_{\mathcal{A}}(a_0)\setminus\{b_0\}$
satisfying $|\mathcal{A}(b_*)\cap C_1|=n$ and
$\mathcal{A}(b_*)\cap C_2\neq\varnothing$.
Then we can pick up $a_k^{(i)}\in \mathcal{A}(b_*)\cap C_2$ from some
fence $\mathcal{F}^{(i)}$, and find a path $(x_0,\ldots,x_{m-1})$ from
$x_0=b_0$ to $x_{m-1} = b_{k-1}^{(i)}$ in $T_0$.
Similarly to the proof of Lemma~\ref{fence.lem}
we can choose $y_i$ from the nonempty subset \eqref{pick.y}
for each $i=1,\ldots,m-1$,
and set
\begin{math}
  y_m = a_k^{(i)}\in
  \mathcal{A}(b_{k-1}^{(i)})\cap\mathcal{A}(b_*)
  \setminus(\mathcal{A}(b_0)\cap\mathcal{A}(b_*)),
\end{math}
which allows us to construct a fence $\mathcal{F}\in\Gamma_0$
supported by $T_0$ from $b_0$ to $b_*$.
Since $|\mathcal{A}(b_*)\cap C_1|=n$,
we obtain
$|\{i:b_*<a_1^{(i)}\}|\ge M+1$,
which contradicts the maximum $M$ of such numbers
held by the fence $\mathcal{F}^{(1)}$.
\end{proof}

Let $\mathcal{A}'$ be the subposet of $\mathcal{A}$ induced
on $\{a_0,b_0\}\cup B_1\cup B_*\cup C_1\cup C_2$,
and let $\mathcal{S}$ be the $(n+2)$-legged W$_{\star}$-poset of Proposition~\ref{main.claim}.
By $\delta_z$ we denote a point mass probability measure on $\{z\}$.
In Lemma~\ref{fail.me}
we introduce a system $(P_\alpha:\alpha\in\mathcal{A}')$ of probability
measures on $\mathcal{S}$ by
\begin{align}
  P_{a_0} &= \frac{1}{n+1}(\delta_{y_0}+n\delta_z);
  \nonumber \\
  P_{b_0} &= \frac{1}{n+1}\left(\delta_{y_0}+\sum_{j=1}^n\delta_{y_j}\right);
  \label{p.b0} \\  
  P_{a_1^{(j)}} &= \frac{1}{n+1}(\delta_{y_j}+n\delta_z)
  \text{ for $a_1^{(j)}\in C_1$, }
  \label{p.a1}
\end{align}
and
\begin{equation}\label{p.alpha}
  P_{\alpha} = \frac{1}{n+1}\left(\delta_{y_*}
  +(n-|\mathcal{A}(\alpha)\cap C_1|)\delta_z
  +\sum\{\delta_{y_j}:a_1^{(j)}\ge\alpha\}\right)
\end{equation}  
for $\alpha\in B_1\cup C_2$;
\begin{equation}\label{p.beta*}
  P_{\beta_*} = \frac{1}{n+1}\left(
  \delta_{y_*}+\delta_{y_0}+(n-1-|\mathcal{A}(\beta_*)\cap C_1|)\delta_z
  +\sum\{\delta_{y_j}:a_1^{(j)}\ge\beta_*\}\right)
\end{equation}
for $\beta_*\in B_*$.
It should be noted
in the construction of $P_{\beta_*}$ that
$|\mathcal{A}(b_*^{(i)})\cap C_1|\le n-1$ for $b_*^{(i)}\in B_*$
by Lemma~\ref{b.star}.

\begin{lemma}\label{fail.me}
The system $(P_\alpha:\alpha\in\mathcal{A}')$ is stochastically monotone,
but it is not realizably monotone.
\end{lemma}

\begin{proof}
First we show the stochastic monotonicity.
Within \eqref{p.a1}
$C_1$ is mutually incomparable, and so are $P_{a_1^{(j)}}$'s.
Within \eqref{p.alpha} or \eqref{p.beta*}
we have $P_{\alpha}\preceq P_{\beta}$ if $\alpha<\beta$ for any pair
$\{\alpha,\beta\}$ belonging to $B_1\cup C_2$ or $B_*$.
Let $a_1^{(j)}\in C_1$, $\alpha\in B_1\cup C_2$ and $\beta_*\in B_*$ be arbitrarily
fixed.
We can find $a_0$ comparable only with $b_0$ and $\beta_*$,
and obtain $P_{b_0}, P_{\beta_*} \preceq P_{a_0}$.
We also obtain $P_{b_0}\preceq P_{a_1^{(j)}}$.
For all the other comparable pairs we can verify
$P_{\alpha}\preceq P_{a_1^{(j)}}$ if $\alpha<a_1^{(j)}$;
$P_{\beta_*}\preceq P_{a_1^{(j)}}$ if $\beta_*<a_1^{(j)}$;
$P_{\beta_*}\preceq P_{\alpha}$ if $\beta_*\le\alpha$.
Thus, we have established the desired monotonicity.

We now prove by contradiction that the system is not realizably monotone.
Suppose that there is a system $(X_\alpha:\alpha\in\mathcal{A}')$
of random variables such that
the joint probability distribution $\mathbb{P}$
realizes the monotonicity and
each $X_\alpha$ is marginally distributed as $P_\alpha$.
Then we can observe that
\begin{enumerate}
\renewcommand{\labelenumi}{(\roman{enumi})}
\item  
$E_0=\{X_{b_0}=y_0\}$ implies $X_{a_0}=y_0$ and $X_{b_*^{(i)}}=y_0$ for each $i=1,\ldots,n$;
\item  
$E_1 = \{X_{b_0}=y_1\}$ implies $X_{a_1^{(1)}}=y_1$ and $X_{b_1^{(1)}}=y_1$;
\item  
$E_*=\{X_{b_*^{(1)}}=y_*\}$ implies $X_{a_k^{(1)}}=y_*$, $k=2,\ldots,l_1+1$,
and $X_{b_1^{(1)}}=y_*$;
\item  
for each $j=2,\ldots,n$,
the event $E_j=\{X_{b_0}=y_j\}$ implies
$X_{a_1^{(j)}}=y_j$ and $X_{b_*^{(1)}}=y_j$ since
$a_1^{(j)}>b_*^{(1)}$.
\end{enumerate}
Therefore,
$E_0=\{X_{b_0}=y_0\}$, $E_1=\{X_{b_0}=y_1\}$ and $E_j=\{X_{b_0}=y_j\}$,
$j=2,\ldots,n$, are mutually disjoint;
$E_1=\{X_{b_1^{(1)}}=y_1\}$, $E_*=\{X_{b_1^{(1)}}=y_*\}$ and
$E_j=\{X_{b_1^{(1)}}=y_j\}$,
$j=2,\ldots,n$, are mutually disjoint;
$E_0=\{X_{b_*^{(1)}}=y_0\}$ and $E_*=\{X_{b_*^{(1)}}=y_*\}$
are disjoint.
Hence, we find $E_*$ and $E_i$, $i=0,\ldots,n$, mutually disjoint.
Consequently we obtain
\begin{equation*}
  \mathbb{P}\left(E_*\cup\bigcup_{i=0}^n E_i\right)
  = \mathbb{P}(E_*) + \sum_{i=0}^n \mathbb{P}(E_i)
  = \frac{n+2}{n+1}
\end{equation*}
which contradicts that $\mathbb{P}$ is a probability measure.
\end{proof}

\begin{example}
Continue Example~\ref{4d.example}
in which we have $\mathcal{A}' = \mathcal{A}$.
Let $\mathcal{S}$ be a $4$-legged W$_{\star}$-poset of Figure~\ref{wposet}
with $n=2$.
Then we can construct the stochastically monotone system of Lemma~\ref{fail.me}
by
\begin{align*}
  &
  P_{a_0} = \frac{1}{3}\delta_{y_0} + \frac{2}{3}\delta_{z};
  & &
  P_{b_0} = \frac{1}{3}(\delta_{y_0} + \delta_{y_1} + \delta_{y_2});
  \\ &
  P_{a_1^{(1)}} = \frac{1}{3}\delta_{y_1} + \frac{2}{3}\delta_{z};
  & &
  P_{a_1^{(2)}} = \frac{1}{3}\delta_{y_2} + \frac{2}{3}\delta_{z};
  \\ &
  P_{b_1} = \frac{1}{3}(\delta_{y_*} + \delta_{y_1} + \delta_{y_2});
  & &
  P_{a_2} = \frac{1}{3}\delta_{y_*} + \frac{2}{3}\delta_{z};
  \\ &
  P_{b_*^{(1)}} = \frac{1}{3}(\delta_{y_*} + \delta_{y_0} + \delta_{y_2});
  & &
  P_{b_*^{(2)}} = \frac{1}{3}(\delta_{y_*} + \delta_{y_0} + \delta_{y_1}).
\end{align*}

Suppose that there exists a system $(X_\alpha:\alpha\in\mathcal{A})$
which realizes the monotonicity.
Then we can observe that
\begin{align*}
  E_0 &= \{X_{b_0}=y_0\} = \{X_{a_0}=y_0\} = \{X_{b_*^{(1)}}=y_0\},
  \\
  E_1 &= \{X_{b_0}=y_1\} = \{X_{a_1^{(1)}}=y_1\} = \{X_{b_1}=y_1\},
  \\
  E_2 &= \{X_{b_0}=y_2\} = \{X_{a_1^{(2)}}=y_2\} = \{X_{b_*^{(1)}}=y_2\},
  \\
  E_* &= \{X_{b_*^{(1)}}=y_*\} = \{X_{a_2}=y_*\} = \{X_{b_1}=y_*\}
\end{align*}
are mutually disjoint,
and therefore, we obtain
$\mathbb{P}(E_0\cup E_1\cup E_2\cup E_*) = 4/3 > 1$,
which is a contradiction.
Hence, the system cannot be realizably monotone.
\end{example}

For each $\alpha\in\mathcal{A}$ we can define a principal down-set of
$\mathcal{A}$ by
\begin{equation*}
  \mathcal{A}^*(\alpha) = \{\beta\in\mathcal{A}:
  \mbox{$\beta\le\alpha$ in $\mathcal{A}$}\}
\end{equation*}
Then we can introduce a partition of $\mathcal{A}$ by
\begin{align*}
  A_0 &= \mathcal{A}^*(a_0)\cap\bigcap_{j=1}^n\mathcal{A}^*(a_1^{(j)});
  \\
  A_1 &= \mathcal{A}\setminus\mathcal{A}^*(a_0);
  \\
  A_* &= \mathcal{A}^*(a_0)\setminus\bigcap_{j=1}^n\mathcal{A}^*(a_1^{(j)}),
\end{align*}
and find $b_0\in A_0$, $C_1\cup C_2\cup B_1\subseteq A_1$, and
$\{a_0\}\cup B_*\subseteq A_*$.
By setting the down-set
\begin{equation*}
  A_2 = \bigcup_{\alpha_2\in C_2}\mathcal{A}^*(\alpha_2)
\end{equation*}
we can further refine the partition of $\mathcal{A}$ by
\begin{equation}\label{part}
  A_0,\,
  A_1\cap A_2,\,
  A_1\setminus A_2,\,
  A_*\cap A_2
  \text{ and }
  A_*\setminus A_2,
\end{equation}
and find $b_0\in A_0$, $B_1\cup C_2\subseteq A_1\cap A_2$,
$C_1\subseteq A_1\setminus A_2$,
$B_*\subseteq A_*\cap A_2$ and
$a_0\in A_*\setminus A_2$.

\begin{lemma}\label{part.a}
The down-sets $A_0$ and $A_2$ are disjoint.
\end{lemma}

\begin{proof}
Suppose that $A_0\cap A_2\neq\varnothing$.
Then we can find a minimal element $b_*\in A_0\cap A_2$
so that $b_*\in A_0\cap D_{\mathcal{A}}$ and
$b_*\le\alpha_2$ for some $\alpha_2\in C_2$.
Since $\alpha_2\in C_2$ and $b_0$ are incomparable,
we also find $b_*\in D_{\mathcal{A}}(a_0)\setminus\{b_0\}$.
Since $b_*\in A_0$, we have $|\mathcal{A}(b_*)\cap C_1| = n$,
but it contradicts that $|\mathcal{A}(b_*)\cap C_1|\le n-1$ by Lemma~\ref{b.star}.
\end{proof}

We can extend $(P_\alpha:\alpha\in\mathcal{A}')$ and construct
a system $(P_\alpha:\alpha\in\mathcal{A})$ of probability measures on
$\mathcal{S}$ by setting
$P_{\beta_0} = P_{b_0}$ of \eqref{p.b0} for $\beta_0\in A_0$;
(\ref{p.alpha}) for $\alpha\in A_1\cap A_2$;
\begin{equation*}
  P_{\alpha_1} = \frac{1}{n+1}\left(
  (n+1-|\mathcal{A}(\alpha_1)\cap C_1|)\delta_z
  +\sum\{\delta_{y_j}:a_1^{(j)}\ge\alpha_1\}\right)
\end{equation*}
for $\alpha_1\in A_1\setminus A_2$;
(\ref{p.beta*}) for $\beta_*\in A_*\cap A_2$;
\begin{equation*}
  P_{\alpha_0} = \frac{1}{n+1}\left(\delta_{y_0}
  +(n-|\mathcal{A}(\alpha_0)\cap C_1|)\delta_z
  +\sum\{\delta_{y_j}:a_1^{(j)}\ge\alpha_0\}\right)
\end{equation*}
for $\alpha_0\in A_*\setminus A_2$.

A stochastically monotone extension of
$(P_\alpha:\alpha\in\mathcal{A}')$ cannot be realizably monotone
by Lemma~\ref{fail.me}.
Thus, the following lemma completes the proof of Proposition~\ref{main.claim}.

\begin{lemma}\label{sm.extension}
The system $(P_\alpha:\alpha\in\mathcal{A})$ is stochastically monotone.
\end{lemma}

\begin{proof}
Clearly we have $P_{\alpha}\preceq P_{\beta}$ if $\alpha<\beta$ for a pair
$\{\alpha,\beta\}$ belonging to one of the refined partition (\ref{part}).
Let
\begin{equation*}
  \beta_0\in A_0,\,
  \alpha\in A_1\cap A_2,\,
  \alpha_1\in A_1\setminus A_2,\,
  \beta_*\in A_*\cap A_2,
  \text{ and }  
  \alpha_0\in A_*\setminus A_2
\end{equation*}
be arbitrarily fixed.
\begin{enumerate}
\renewcommand{\labelenumi}{(\roman{enumi})}
\item
Since $A_1\setminus A_2$ is an up-set of $\mathcal{A}$,
a comparable pair with $\alpha_1$ satisfies
$\alpha_0,\beta_0,\alpha,\beta_*<\alpha_1$,
for which we see
$P_{\alpha_0},P_{\beta_0},P_{\alpha},P_{\beta_*}\preceq P_{\alpha_1}$.
\item
Consider a comparable pair with $\alpha_0$
from $\beta_0$, $\alpha$ or $\beta_*$.
We observe that
$\alpha<\alpha_0$ cannot happen; otherwise,
$\alpha_0\in A_*$ implies $\alpha<a_0$ and
contradicts $\alpha\in A_1$.
Similarly $\alpha_0<\alpha$ cannot happen because
$\alpha\in A_2$ implies
$\alpha_0<\alpha_2$ for some $\alpha_2\in C_2$,
contradicting $\alpha_0\not\in A_2$.
Hence, the pair $\{\alpha_0,\alpha\}$ is incomparable.
Since $A_0$ is a down-set of $\mathcal{A}$, a comparable pair $\{\alpha_0,\beta_0\}$
satisfies $\beta_0<\alpha_0$.
A comparable pair $\{\alpha_0,\beta_*\}$ must have $\beta_*<\alpha_0$;
otherwise, $\beta_*\in A_2$ implies $\alpha_0<\alpha_2$ for some $\alpha_2\in C_2$,
contradicting $\alpha_0\not\in A_2$.
For these comparable pairs we have
$P_{\beta_0},P_{\beta_*}\preceq P_{\alpha_0}$ if $\beta_0,\beta_*<\alpha_0$.
\item
By Lemma~\ref{part.a} we find $\beta_0$ incomparable with $\alpha$ and
$\beta_*$ since $\beta_0\in A_0$ and $\alpha,\beta_*\in A_2$.
Thus, there is no pair with $\beta_0$ not considered in (i)--(ii).
\item
A comparable pair $\{\alpha,\beta_*\}$ must satisfy $\beta_*<\alpha$;
otherwise,
$\beta_*\in A_*$ implies $\alpha<a_0$, contradicting $\alpha\in A_1$.
Then we have $P_{\beta_*}\preceq P_{\alpha}$ if $\beta_*<\alpha$.
\end{enumerate}
Therefore, we have established the desired monotonicity for all the comparable pairs.
\end{proof}

\bibliographystyle{plain}
\bibliography{monotonicity}

\end{document}